\documentclass[12pt]{amsart}
\usepackage{amsthm,amsmath,amssymb,amsopn}
\usepackage{graphicx,epstopdf,epsfig,subfig}
\usepackage{booktabs}
\usepackage{colortbl,multirow}
\usepackage{appendix}
\numberwithin{equation}{section}
\numberwithin{table}{section}
\numberwithin{figure}{section}

\newtheorem{lemma}{Lemma}[section]
\newtheorem{theorem}{Theorem}[section]

\newtheorem{remark}{Remark}[section]

\theoremstyle{definition}

\renewcommand{\div}[1]{\nabla \cdot #1} 

\newcommand{\wt}[1]{\widetilde{#1}}
\newcommand{\wh}[1]{\widehat{#1}}

\newcommand{\MTh}{\mathcal{T}_h}

\newcommand{\lj}{[ \hspace{-2pt} [}
\newcommand{\rj}{] \hspace{-2pt} ]}
\def\al{\alpha}
\def\eps{\varepsilon}
\def\na{\nabla}
\def\pa{\partial}
\def\x{\times}
\def\Lam{\Lambda}
\def\Om{\Omega}
\newcommand{\mb}[1]{\mathbb{#1}}
\newcommand{\mc}[1]{\mathcal{#1}}
\newcommand{\abs}[1]{\left\lvert#1\right\rvert}
\newcommand{\nm}[2]{\|\,#1\,\|_{#2}}
\newcommand{\snm}[2]{\abs{\,#1\,}_{#2}}
\newcommand{\Lr}[1]{\left(#1\right)}
\newcommand{\jump}[1]{[\![#1]\!]}
\newcommand{\aver}[1]{\left\{\!\{#1\right\}\!\}}
\newcommand{\set}[2]{\{\,#1\,\mid\,#2\}}
\newcommand{\enernm}[1]{\|\!|\,#1\,\|\!|}

\DeclareMathOperator{\argmin}{arg min}

\def\dx{\,\mathrm{d}x}
\def\ds{\mathrm{d}s}

\definecolor{lightgray}{gray}{0.9}

\begin{document}
\title[DG Method with One Unknown Per Element]{An Arbitrary-Order Discontinuous Galerkin Method with One Unknown Per Element}
\author[R. Li]{Ruo Li}
\address{CAPT, LMAM and School of Mathematical Sciences, Peking
University, Beijing 100871, P.R. China} \email{rli@math.pku.edu.cn}

\author[P.-B. Ming]{Pingbing Ming}
\address{LSEC, Academy of Mathematics and Systems Science, Chinese Academy of Sciences,
  No.55, Zhong-Guan-Cun East Road, Beijing, 100190, P.R. China}
\address{School of Mathematical Sciences, University of Chinese Academy of Sciences,
  No. 19A, Yu-Quan Road, Beijing, 100049, P.R. China}\email{mpb@lsec.cc.ac.cn}

\author[Z.-Y. Sun]{Zhiyuan Sun}
\address{School of Mathematical Sciences, Peking University, Beijing
  100871, P.R. China} \email{zysun@math.pku.edu.cn}

\author[Z.-J. Yang]{Zhijian Yang}
\address{School of Mathematics and Statistics, Wuhan University, Wuhan
  430072, P.R. China} \email{zjyang.math@whu.edu.cn}
\date{\today}
\maketitle

\begin{abstract}

  We propose an arbitrary-order discontinuous Galerkin method for second-order elliptic problem on
  general polygonal mesh with only one degree of freedom per element. This is
  achieved by locally solving a discrete least-squares over a
  neighboring element patch. Under a geometrical condition on the element patch, we
  prove an optimal a priori error estimates for the energy norm and for the L$^2$ norm. The
  accuracy and the efficiency of the method up to order six on several polygonal meshes are illustrated by a set
  of benchmark problems.



\end{abstract}
\section{Introduction}\label{sec:intro}
The discontinuous Galerkin method (DG)~\cite{ReedHill:1973, cockburn2000development} is by now a very standard numerical method to simulate a wide variety of partial differential equations of scientific or engineer interest. Recently there are quite a few work concerning the discontinuous Galerkin method on general polytopic (polygonal or polyhedral)
meshes~\cite{hesthaven2007nodal,Ern:2012, antonietti2013hp, Lipnikov:2013, cangiani2014hp, Wirasaet:2014, Cangiani:2016, Houston:2017}. Unlike the finite element method~\cite{sukumar2004conforming}, the full discontinuity across the element interfaces of the trial and test function spaces of DG method lends itself naturally to the polytopic meshes, which does provide more flexibility in implementation, in particular for domain with microstructures or problems with certain physical constraints. Such meshes may ease the triangulation of complex domains or domains with microstructures. On the other hand, compared to the classical conforming finite element method, the DG method is computationally expensive over a particular computational mesh and approximation order, which is due to the fact that the rapid increasing of
the local degrees of freedom. Moreover, for certain problems, such as fluid solid interaction problems or a heat diffusion problem that is coupled or is embed into a compressible fluid problem, there is no enough information to support a large number of local degrees of freedom. Though this problem is partially solved by the
agglomeration-based physical frame DG method~\cite{Bassi:2012a, Bassi:2014} and the hybrid DG method~\cite{Cockburn:2009}, it is desirable to develop a DG method with less local degrees of freedom while retaining high accuracy.

A common trait is to employ patch reconstruction to achieve high
accuracy. To the best of the authors' knowledge, such reconstruction idea can be traced back
to the endeavors on developing three node plate bending elements and
simple shell elements in the early 1970s; see, for example~\cite{NayUtku:1972,
  HampshireTopping:1992, PhaalCalladine:1992,OnateZarate:1993}. Similar
ideas may also be found in WENO~\cite{Venka:1995} and finite volume method for hyperbolic
conservation laws~\cite{BarthLarson:2002}.

This motivates us to use patches of a piecewise constant
function to reconstruct a piecewise high order polynomial on each
element, which is achieved by solving a discrete least-squares approximation problem over
element patch. Such approach has been used in~\cite{li2012efficient} to reconstruct piecewise effective tensor field from scattered data for the multiscale partial differential equations. This new space is a sub-space of the commonly used
finite element space corresponding to DG methods. This new finite element space may be combined with any other DG formulations to numerically solve even more general elliptic problems such as plate bending problem~\cite{Sun:2017}, Stokes flow problems, and eigenvalue problems, just name a few. As a starting point, we employ the Interior Penalty discontinuous Galerkin (IPDG) method~\cite{arnold1982interior} with this reconstructed finite element space to solve Poisson problem. Under a mild condition on the geometry of the element patch, we proved that this
reconstructed finite element space admits optimal approximation properties in certain broken Sobolev norms, by which we proved the optimal error estimates in the DG-energy norm and in the L$^2$ norm of the proposed method.

Our method possesses several attractive features. First, arbitrary order accuracy may be
achieved with increasing the order of the reconstruction, while there is only one degree of freedom per element, by contrast to the standard DG method, which requires at least three unknowns on each element. From this aspect of view, the proposed method has a flavor of finite volume method. Second, the method can be used on any shape of elements, which may be triangles, quadrilaterals, polygons in two dimension, or tetrahedron, prism, pyramid, hexahedron
in three dimension. In particular, the method may be used on the
hybrid mesh, which is nowadays quite common in simulations because it
can handle the complicated domain or even reduce the total number of
unknowns~\cite{Yamakawa:2009}. Third, the reconstruction procedure of the proposed method is stable with respect to the small perturbation of the data, which is of practical interesting due to the measurement error. Our results for the reconstruction procedure is of independent interest for the discrete least-squares~\cite{reichel1986Polyappbyls}.

A closely related approach is a special DG method proposed in~\cite{Larsson:2012}. The authors introduced a family of
continuous linear finite elements for the Kirhhoff-Love plate model. A continuous linear interpolation of the deflection field is employed to
reconstruct a discontinuous quadratic deflection field by solving a local least-squares problem over element patch. It is worth mentioning that one of their reconstruction method is the same with the second order constrained reconstruction in~\cite{li2012efficient}. Moreover, this method only applies to structured mesh. Another closely related method is the cell-centered Galerkin method presented in~\cite{DiPierto:2014}. The authors developed an arbitrary-order discretization method of diffusion problems on general polyhedral mesh. The cornerstone
of this method is the locally reconstructed discrete gradient operator and a stabilized term, while our method is the locally reconstructed finite element space. This method has also been successfully extended to mixed form~\cite{DiPierto:2017} recently.

The rest of the paper is organized as follows. In \S~\ref{sec:basis}, we describe
the reconstruction finite element space and prove its approximation
properties and the stability properties of the reconstruction procedure. In \S~\ref{sec:weakform}, we present the
interior penalty discontinued Galerkin method for Poisson problem with
the reconstructed approximation space and prove a priori error
estimate, and in \S~\ref{sec:examples}, numerical results for
second order elliptic problems in two dimension are presented. Finally, in \S~\ref{sec:conclusion}, we summarize
the work and draw some conclusions.

Throughout this paper, we shall use standard notations for Sobolev
spaces, norms and seminorms, cf.~\cite{AdamsFournier:2003}; see, for example,
\(\|u\|_{H^1(D)}{:}=\nm{u}{L^2(D)}+\nm{\na u}{L^2(D)}\) for any bounded domain $D$. We use $C$ as
a generic constant independent of the mesh size, which may change from
line to line.  We mainly focus on two dimensional problem though most
results are valid in three dimension.
\section{Approximation Space}\label{sec:basis}
Let $\Om$ be a polygonal domain in $\mb{R}^2$. The mesh
$\MTh$ is a triangulation of $\Om$ with polygons $K$, which may not be convex. Here $h{:}=\max_{K\in\MTh}h_K$ with $h_K$ the diameter of $K$. We denote $\abs{K}$ the area of $K$. Let $\mb{P}_n(D)$ be a set of polynomial in two variables with total degree at most $n$ confined to domain $D$, where $D$ may be an element $K$ or an agglomeration of the elements belong to $\MTh$. We assume that the mesh $\MTh$ satisfies the
following shape regularity conditions, which were introduced originally in~\cite{Brezzi:2009} to study the convergence of mimetic finite difference. Detailed discussion on such conditions can be found in~\cite[\S 1.6]{DaVeiga2014}.

There exist
\begin{enumerate}
\item an integer number $N$ independent of $h$;
\item a real positive number $\sigma$ independent of $h$;
\item a compatible sub-decomposition $\wt{\MTh}$
\end{enumerate}
such that
\begin{enumerate}
\item[{\bf A1}\;]any element $K$ admits a sub-decomposition
  $\wt{\MTh}|_K$ that consists of at most $N$ triangles $T$.
\item[{\bf A2}\;]Any $T\in\wt{\MTh}$ is shape-regular in the sense of
  Ciarlet-Raviart~\cite{ciarlet:1978}: there exists $\sigma$ such that
  $h_T/\rho_T\le\sigma$, where $\rho_T$ is the radius of the largest
  ball inscribed in $T$.
\end{enumerate}

Assumptions {\bf A1} and {\bf A2} impose quite weak constraints on the triangulation, which may contain elements with quite general shapes, for example, non-convex or degenerate elements are allowed.

The above shape regularity assumptions lead to some useful consequences,
which will be extensively used in the later analysis.
\begin{enumerate}
\item[{\bf M1}]For any $T\in\wt{\MTh}$, there exists $\rho_1\ge 1$
  that depends on $N$ and $\sigma$ such that $h_K/h_T\le\rho_1$.

\item[{\bf M2}][{\it Agmon inequality}]\;There exists $C$ that depends
  on $N$ and $\sigma$, but independent of $h_K$ such that
\begin{equation}\label{eq:agmon}
\nm{v}{L^2(\pa K)}^2\le C\Lr{h_K^{-1}\nm{v}{L^2(K)}^2+h_K\nm{\na
    v}{L^2(K)}^2}\qquad\text{for all\quad}v\in H^1(K).
\end{equation}

\item[{\bf M3}][{\it Approximation property}]\;There exists $C$ that
  depends on $N,r$ and $\sigma$, but independent of $h_K$ such that for
  any $v\in H^{r+1}(K)$, there exists an approximation polynomial
  $\wt{v}\in\mb{P}_r(K)$ such that
\begin{equation}\label{eq:app}
\nm{v-\wt{v}}{L^2(K)}+h_K\nm{\na(v-\wt{v})}{L^2(K)}\le
Ch_K^{r+1}\snm{v}{H^{r+1}(K)}.
\end{equation}

\item[{\bf M4}][{\it Inverse inequality}]\;For any $v\in\mb{P}_r(K)$,
  there exists a constant $C$ that depends only on $N,r,\sigma$ and $\rho_1$
  such that
\begin{equation}\label{eq:inverse}
\nm{\na v}{L^2(K)}\le Ch_K^{-1}\nm{v}{L^2(K)}.
\end{equation}
\end{enumerate}

Note that we have not list all the mesh conditions in~\cite{Brezzi:2009} and~\cite[\S 1.6]{DaVeiga2014}, while the above two assumptions {\bf A1} and {\bf A2} suffice for our purpose. The {\em Agmon inequality} {\bf M2} and the {\em Approximation property} {\bf M3} have been proved in~\cite[\S 1.6.3]{DaVeiga2014}. As to {\bf M3}, one may take the approximation polynomial $\wt{v}$ as the averaged Taylor polynomial of order $r+1$~\cite{brenner2007mathematical}. The inverse inequality {\bf M4} can be proved as follows. For any $v\in\mb{P}_r(K)$, the restriction $v|_T\in\mb{P}_r(T)$, using the following standard inverse inequality on triangle $T$~\cite{ciarlet:1978}, we have
\[
\nm{\na v}{L^2(T)}\le C_{\text{inv}}h_T^{-1}\nm{v}{L^2(T)}.
\]
where $C_{\text{inv}}$ depends on $\sigma$ and $r$ while is independent of $h_T$. Summing up all $T\in\wt{T}_h|_K$, and using {\bf M1}, we obtain
\begin{align*}
\nm{\na v}{L^2(K)}^2&=\sum_{T\subset\wt{\mc{T}}_h|_K}\nm{\na v}{L^2(T)}^2\le C_{\text{inv}}^2\rho_1^2 h_K^{-2}\sum_{T\subset\wt{\mc{T}}_h|_K}\nm{v}{L^2(T)}^2\\
&=C_{\text{inv}}^2\rho_1^2 h_K^{-2}\nm{v}{L^2(K)}^2.
\end{align*}
This gives~\eqref{eq:inverse} with $C=C_{\text{inv}}\rho_1$.

A combination of~\eqref{eq:agmon} and~\eqref{eq:inverse} yields the {\em discrete trace inequality}: for any $v\in\mb{P}_r(K)$,
\begin{equation}\label{eq:traceinv}
\nm{v}{L^2(\pa K)}\le C\Lr{1+C_{\text{inv}}\rho_1}h_K^{-1/2}\nm{v}{L^2(K)}.
\end{equation}
\begin{remark}
The above four inequalities~\eqref{eq:agmon},~\eqref{eq:app},~\eqref{eq:inverse} and~\eqref{eq:traceinv} are
the foundation to derive the error estimate for the IPDG method~\cite{arnold1982interior}, which are also valid over the polygonal meshes satisfying different shape regular conditions; see, for example~\cite[\S 1.4]{Ern:2012} and~\cite{Mu:2014}).
\end{remark}
\subsection{Reconstruction operator}
Given the triangulation $\MTh$, we define the reconstruction operator in a piecewise manner as
follows. For each element $K\in\MTh$, we firstly assign a sampling node
$x_K \in K$ that is preferably in the interior of the element $K$, and construct an element patch $S(K)$. The element patch $S(K)$ usually contains $K$ and some elements around $K$. There are many different ways to find the
sampling nodes and construct the element patch. For example, we may let the barycenter of the element $K$ as the
sampling node, while it can be more flexible due to the stability property of the least-squares reconstruction; cf. Lemma~\ref{lema:perturb}. The element patch may be built up
in the following two ways. The first way is that we initialize $S(K)$ as $K$, and add all the Moore neighbors (elements with nonempty intersection with the closure of $K$)~\cite{Sullivan:2001} into $S(K)$ recursively until sufficiently large number of elements are collected into the element patch. Such kind of construction has been used in~\cite{li2012efficient} to reconstruct effective tensor field from scattered data. The second way is the same with the first one except that we use the Von Neumann neighbor (adjacent edge-neighboring elements)~\cite{Sullivan:2001} instead of Moore neighbor. An example for such $S(K)$ is shown in Figure~\ref{patch_ponits}. We denote by $t$ the recursion depth for the element patch $S(K)$.
\begin{figure}
  \begin{center}
    \includegraphics[width=0.3\textwidth]{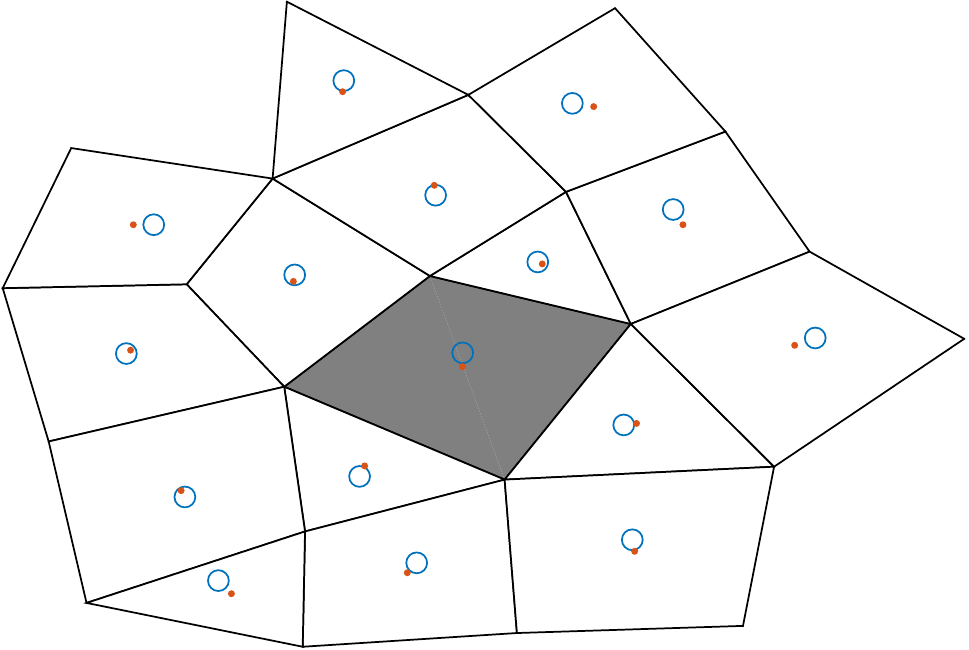}
    \includegraphics[width=0.3\textwidth]{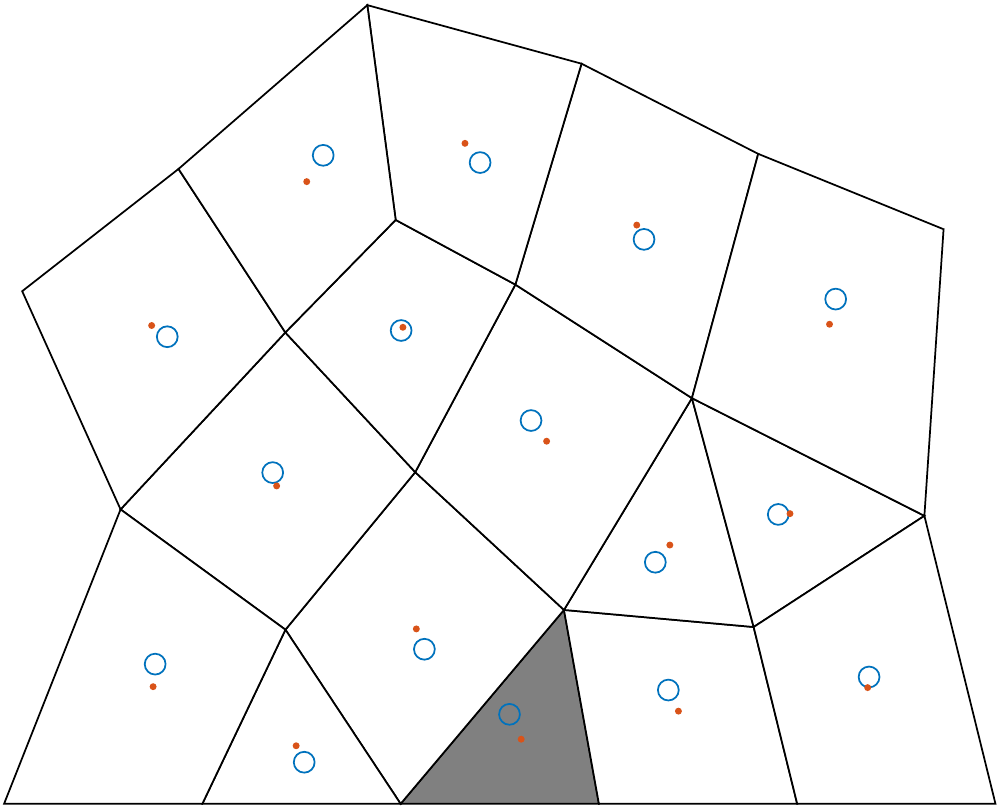}
    \includegraphics[width=0.3\textwidth]{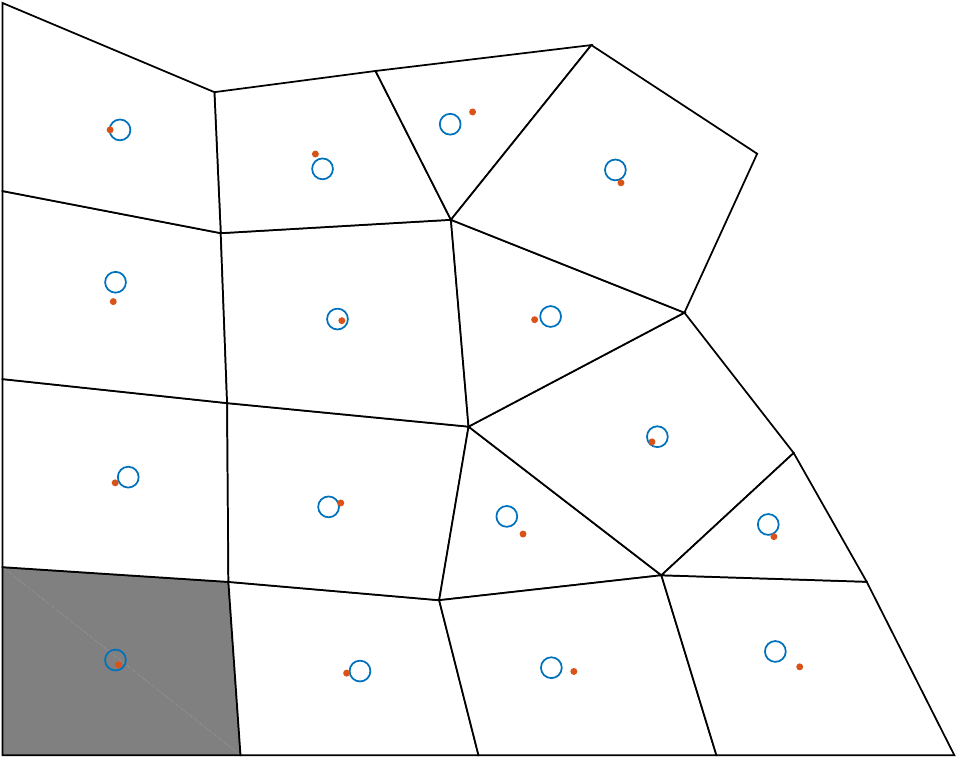}
    \caption{The element patches $S(K)$ in the interior of the domain
      (left), along the boundary of the domain (middle) and at the corner of the domain (right),
      together with randomly perturbed sampling points for Example 3
      in \S~\ref{sec:examples}. Here the element $K$ is marked in
      black and the sampling nodes in $\mc{I}(K)$ are the barycenters of the elements, which are marked in dots, and the
      perturbed sampling nodes in $\wt{\mc{I}}(K)$ are marked as small circles.}
    \label{patch_ponits}
  \end{center}
\end{figure}

We denote $\mc{I}(K)$ the set of the sampling nodes belonging to
$S(K)$ with $\#\mc{I}(K)$ its cardinality, and let $\#S(K)$ be the
number of elements belonging to $S(K)$. These two numbers are equal. We
define $d_K{:}=\text{diam}\;S(K)$ and $d=\max_{K\in\MTh}d_K$. Moreover, we assume that $S(K)$
satisfies the following geometrical assumption.
\noindent\vskip .5cm {\bf Assumption A}\; For every $K\in\MTh$, there
exist constants $R$ and $r$ that are independent of $K$ such that
$B_r\subset S(K)\subset B_R$ with $R\ge 2r$, and $S(K)$ is star-shaped
with respect to $B_r$, where $B_\rho$ is a disk with radius $\rho$.

As a direct consequence of the above assumption, we have the following
characterization of $S(K)$.
\begin{lemma}\label{lema:geometry}
If {\bf Assumption A} is valid, then for all element $K\in\MTh$, the
element patch $S(K)$ satisfies an interior cone condition, and
there exists a uniform bound $\gamma$ for the chunkiness parameter of $S(K)$.
\end{lemma}

\begin{proof}
By~\cite[Proposition 2.1]{Nar:2005}, if {\bf Assumption A} holds true,
then the element patch $S(K)$ satisfies an interior cone
condition with radius $r$ and angel $\theta=2\arcsin\dfrac{r}{2R}$.

By definition~\cite[Definition 4.2.16]{brenner2007mathematical}, the
chunkiness parameter $\gamma_K$ is defined as the ratio between the diameter of
$S(K)$ and the radius of the largest ball to which $S(K)$ is
star-shaped. This leads to the following bound
\[
\gamma_K{:}=\dfrac{d_K}{r}\le\dfrac{2R}{r}.
\]
Let $\gamma{:}=2R/r$, we obtain a uniform bound on the chunkiness parameter for all
$K\in\MTh$.
\end{proof}

Let $U_h$ be the piecewise constant space associated with $\MTh$,
i.e.,
\[
U_h{:}=\set{v\in L^2(\Om)}{v|_K\in\mb{P}_0(K)}.
\]
For any $v\in U_h$ and for any $K\in\MTh$, we reconstruct a
high order polynomial $\mc{R}_K v$ of degree $m$ by solving the
following discrete least-squares.
\begin{equation}\label{eq:leastsquares}
\mc{R}_K
v=\argmin_{p\in\mb{P}_m(S(K))}\sum_{x\in\mc{I}(K)}\abs{v(x)-p(x)}^2.
\end{equation}
A global reconstruction operator $\mc{R}$ is defined by
$\mc{R}|_K=\mc{R}_K$. Given $\mc{R}$, we embed $U_h$ into a
discontinuous finite element space with piecewise polynomials of order $m$, and denote
$V_h=\mc{R}U_h$.

In what follows, we make the following assumption on the sampling node set
$\mc{I}(K)$.
\noindent\vskip .5cm {\bf Assumption B}\; For any $K\in\MTh$ and
$p\in\mb{P}_m(S(K))$,
\begin{equation}\label{assumption:uniqueness}
p|_{\mc{I}(K)}=0\quad\text{implies\quad}p|_{S(K)}\equiv 0.
\end{equation}

This assumption implies the uniqueness, and equally the existence of
the discrete least-squares~\eqref{eq:leastsquares}. {\bf Assumption B} requires that $\#\mc{I}(K)$ cannot be too small, which is at least $(m+2)(m+1)/2$ to guarantee the unisolvability of the discrete least-squares. A quantitative version of this assumption is
\[
\Lambda(m, \mc{I}(K))<\infty
\]
with
\begin{equation}\label{eq:cons}
\Lambda(m, \mc{I}(K)){:}=\max_{p\in\mb{P}_m(S(K))}
\dfrac{\nm{p}{L^\infty(S(K))}}{\nm{p|_{\mc{I}(K)}}{\ell_\infty}}.
\end{equation}

The reconstruction procedure is robust with respect to the small
perturbation of the sampling nodes due to the following stability
result. In particular, $\Lambda(m, \mc{I}(K))$ remains bounded with respect to small perturbation.
This problem is of practical interest because both the sampled
values and the positions of the sampling nodes are affected by the
measurable errors.
\begin{lemma}\label{lema:perturb}
Let $S(K)$ be the element patch defined above and $g\in
C^1(S(K))$. If we assume that
\begin{enumerate}
\item There exists $\alpha>0$ such that
\begin{equation}\label{eq:recstable1}
\nm{g}{L^\infty(S(K))}\le\alpha\nm{g|_{\mc{I}(K)}}{\ell_\infty}.
\end{equation}

\item $S(K)$ admits a Markov-type inequality in the sense that there
  exists $\beta>0$ such that
\begin{equation}\label{eq:markovgeneral}
\nm{\na g}{L^\infty(S(K))}\le\beta\nm{g}{L^\infty(S(K))}.
\end{equation}
\end{enumerate}

Then for any $\delta\in(0,1)$, there exists
$\eps=\delta/(\alpha\beta)$ such that for any sampling node set
$\wt{\mc{I}}(K)$ that is a perturbation of $\mc{I}(K)$ in the sense
that $\wt{\mc{I}}(K)\subset\mc{I}(K)+B(0,\eps)$ with $B$ centers at
$0$ with radius $\eps$, the following stability estimate is valid:
\begin{equation}\label{eq:perturbstab}
\nm{g}{L^\infty(S(K))}\le\dfrac\alpha{1-\delta}\nm{g|_{\wt{\mc{I}}(K)}}{\ell_\infty}.
\end{equation}
\end{lemma}

This stability estimate for the reconstruction procedure depends on the assumptions~\eqref{eq:recstable1} and~\eqref{eq:markovgeneral}. The validity of these two assumptions hinges on certain geometrical condition on the element patch. For example, if the element patch $S(K)$ is convex, then by~\cite[Lemma 3.5]{li2012efficient}, the first assumption~\eqref{eq:recstable1} is valid with $\al=2$, and the second assumption~\eqref{eq:markovgeneral} is valid with $\beta=4m^2/w(K)$ due to the {\em Markov inequality} of {\sc Wilhelmsen}~\cite{Wilhelmsen:1974} for the convex domain, where $w(K)$ is the width of the convex set $S(K)$. A more general assumption for the validity of these two assumptions may be found in Remark~\ref{remark:convex}.
\begin{proof}
Let $x^\ast\in\mc{I}(K)$ satisfy
$\abs{g(x^\ast)}=\nm{g|_{\mc{I}(K)}}{\ell_\infty}$. There exists
$\wt{x}\in\wt{\mc{I}}(K)$ such that $\abs{\wt{x}-x^\ast}\le\eps$. By
Taylor's expansion,
\begin{align*}
\abs{g(x^\ast)}&\le\abs{g(\wt{x})}+\eps\max_{x\in S(K)}\abs{\na
  g(x)}\\ &\le\nm{g|_{\wt{\mc{I}}(K)}}{\ell_\infty}+\eps\beta\nm{g}{L^\infty(S(K))},
\end{align*}
where we have used the Markov's
inequality~\eqref{eq:markovgeneral}. Therefore, we obtain
\begin{align*}
\nm{g}{L^\infty(S(K))}&\le\alpha\nm{g|_{\mc{I}(K)}}{\ell_\infty}=\al\abs{g(x^\ast)}\\ &\le\al\nm{g|_{\wt{\mc{I}}(K)}}{\ell_\infty}+\eps\al\beta\nm{g}{L^\infty(S(K))}\\ &=\al\nm{g|_{\wt{\mc{I}}(K)}}{\ell_\infty}+\delta\nm{g}{L^\infty(S(K))},
\end{align*}
which immediately implies the stability estimate~\eqref{eq:perturbstab}.
\end{proof}

The following properties of the reconstruction operator $\mc{R}_K$ is
proved in~\cite[Theorem 3.3]{li2012efficient}, which is of vital
importance to our error estimate.
\begin{lemma}\label{theorem:localapp}
If {\em Assumption B} holds, then there exists a unique solution
of~\eqref{eq:leastsquares} for any $K\in\MTh$. The unique solution is
denoted by $\mc{R}_K v$.

Moreover $\mc{R}_K$ satisfies
\begin{equation}\label{eq:invariance}
\mc{R}_Kg=g\quad\text{for all\quad}g\in\mb{P}_m(S(K)).
\end{equation}
The stability property holds true for any $K\in\MTh$ and $g\in
C^0(S(K))$ as
\begin{equation}\label{eq:continuous}
\nm{\mc{R}_K g}{L^{\infty}(K)}\le\Lambda(m , \mc{I}(K)) \sqrt{\#
  \mc{I}(K)}\nm{g|_{\mc{I}(K)}}{\ell_\infty},
\end{equation}
and the quasi-optimal approximation property is valid in the sense
that
\begin{equation}\label{eq:approximation}
\nm{g -\mc{R}_K g}{L^{\infty}(K)}\le\Lambda_m
\inf_{p\in\mb{P}_m(S(K))} \nm{g - p}{L^{\infty}(S(K))},
\end{equation}
where $\Lambda_m{:}=\max_{K\in \MTh}
\{1+\Lambda(m,\mc{I}(K))\sqrt{\# \mc{I}(K)}\}$.
\end{lemma}

By the above lemma, we conclude that the reconstruction $\mc{R}_K g$
is a nearly optimal uniform approximation polynomial to $g$ provided
that $\Lam(m,\mc{T}(K))\sqrt{\#\mc{I}(K)}$ can be bounded. As a direct
consequence of this property, we shall prove below that such nearly
optimal approximation property of the reconstruction is also valid
with respect to the broken H$^1$-norms.
\begin{lemma}
If {\em Assumption B} holds, then there exists $C$
that depends on $N,\sigma$ and $\gamma$ such that
\begin{align}
\nm{g-\mc{R}_{K} g}{L^{2}(K)}&\le C\Lam_m
h_Kd_K^m\snm{g}{H^{m+1}(S(K))}.\label{eq:l2app}\\ \nm{\na(g-\mathcal{R}_K
  g)}{L^2(K)}&\le
C\Lr{h_K^m+\Lam_{m}d_K^m}\snm{g}{H^{m+1}(S(K))}.\label{eq:h1app}
\end{align}
\end{lemma}

\begin{proof}
By Lemma~\ref{lema:geometry}, the element patch is star-shaped
with respect to a disk $B_r$ with a uniform chunkness parameter,
using~\cite[Theorem 3.2]{DupontScott:1980}, we take $p=Q^{m+1}g\in\mb{P}_m$ with
$Q^{m+1}g$ the averaged Taylor polynomial of order $m+1$ in the right-hand side
of~\eqref{eq:approximation}, then
\begin{equation}\label{eq:starapp}
\inf_{p
  \in\mb{P}_m(S(K))}\nm{g-p}{L^{\infty}(S(K))}\le\nm{g-Q^{m+1}g}{L^{\infty}(S(K))}\le
Cd_K^m\snm{g}{H^{m+1}(S(K))},
\end{equation}
where $C$ depends on $N,m,\sigma$ and $\gamma$.

Substituting the above estimate~\eqref{eq:starapp}
into~\eqref{eq:approximation}, we obtain
\[
\nm{g-\mc{R}_{K} g}{L^{2}(K)}\le\abs{K}^{1/2}\|g-\mc{R}_{K}
g\|_{L^{\infty}(K)} \le C\Lam_m h_Kd_K^m\snm{g}{H^{m+1}(S(K))}.
\]
This gives~\eqref{eq:l2app}.

Next, let $\wh{g}_m$ be the approximation polynomial
in~\eqref{eq:app} for function $g$, using the {\em inverse
  inequality}~\eqref{eq:inverse} and the approximation
estimate~\eqref{eq:l2app}, we obtain
\begin{align*}
\nm{\na(g-\mathcal{R}_K g)}{L^2(K)}&\le\nm{\na(g-\wh{g}_m)}{L^2(K)}
+\nm{\na(\wh{g}_m-\mathcal{R}_K g)}{L^2(K)}\\ &\le
Ch_K^m\snm{g}{H^{m+1}(K)}+Ch_K^{-1}\nm{\wh{g}_m-\mc{R}_{K}
  g}{L^{2}(K)}\\ &\le
Ch_K^m\snm{g}{H^{m+1}(K)}+Ch_K^{-1}\nm{g-\wh{g}_m}{L^2(K)}
+Ch_K^{-1}\nm{g-\mc{R}_{K} g}{L^{2}(K)}\\ &\le
C\Lr{h_K^m+\Lam_{m}d_K^m}\snm{g}{H^{m+1}(S(K))}.
\end{align*}
This gives~\eqref{eq:h1app} and completes the proof.
\end{proof}
\begin{remark}
If $S(K)$ is convex, then the constants $C$ in~\eqref{eq:l2app} and~\eqref{eq:h1app} are independent of the
chunkness parameter $\gamma$ as proven in~\cite{Dekel:2004}.
\end{remark}

The above lemma indicates that the approximation accuracy of the reconstruction procedure boils down to the
boundedness of $\Lam_m$. We shall seek for conditions of the
triangulation $\MTh$, under which $\Lam_m$ is uniformly
bounded. The authors in~\cite{li2012efficient} have proved
that if the element patch $S(K)$ is convex and the mesh
triangulation is quasi-uniform, then $\Lam_m$ is uniformly
bounded. However, both conditions are not so realistic in
implementation. In next lemma, we shall show that the {\bf Assumption A} is
more suitable in practice, under which $\Lam(m,\mc{I}(K))$ is also
uniformly bounded.
\begin{lemma}\label{lema:bd}
If {\em Assumption A} holds, then for any $\eps>0$, if $r>m\sqrt{2Rh_K(1+1/\eps)}$, then we may
take $\Lam(m,\mc{I}(K))$ as
\begin{equation}\label{eq:uppbd}
\Lam(m,\mc{I}(K))=1+\eps.
\end{equation}
Moreover, if $r>2m\sqrt{Rh_K}$, we may take $\Lam(m,\mc{I}(K))=2$.
\end{lemma}

If {\bf Assumption A} is valid, we usually have $R\simeq th_K$. The above result suggests that $r\simeq m\sqrt{t} h_K$ for the uniform boundedness of $\Lam(m,\mc{I}(K))$.

By~\cite[Theorem 1.2.2.2. and Corollary 1.2.2.3]{Grisvard:1985}, any
convex domain satisfies the uniform cone property. Therefore,
Lemma~\ref{lema:bd} generalizes the corresponding result
in~\cite[Lemma 3.5]{li2012efficient} because it applies to more
general element patch.
\noindent\vspace{.5cm}

{\em Proof of Lemma~\ref{lema:bd}\;}
Let $x^\ast\in\overline{S(K)}$ such that $\abs{p(x^\ast)}=\max_{x\in\overline{S(K)}}\abs{p(x)}$, and
$x_{\ell}\in\mc{I}(K)$ such that $\abs{x_{\ell}-x^\ast}=\min_{y\in\mc{I}(K)}\abs{x^\ast-y}$. Then
\[
\abs{x_{\ell}-x^\ast}\le h_K/2.
\]
By Taylor's expansion, we have
\[
p(x_{\ell})=p(x^\ast)+(x_\ell-x^\ast)\cdot\na p(\xi_x)
\]
with $\xi_x$ a point on the line with end points $x_\ell$ and
$x^\ast$. This gives
\[
\abs{p(x^\ast)}\le\abs{p(x_\ell)}+\dfrac{h_K}{2}\max_{x\in
  S(K)}\abs{\na p(x)}.
\]

By Lemma~\ref{lema:geometry}, the element patch $S(K)$ satisfies
an interior cone condition with radius $r$ and aperture
$\theta=2\arcsin(r/2R)$. By~\cite[Proposition 11.6]{Wendland:2005}, we
have the following Markov inequality:
\begin{equation}\label{eq:markovcone}
\nm{\na
  p}{L^\infty(S(K))}\le\dfrac{2m^2}{r\sin\theta}\nm{p}{L^\infty(S(K))}\quad\text{for
  all\quad}p\in\mb{P}_m(S(K)).
\end{equation}
Using the fact that $\theta/2\le\pi/6$, we have
\[
\sin\theta=2\sin\dfrac{\theta}{2}\cos\dfrac{\theta}{2}=2\dfrac{r}{2R}\cos\dfrac{\theta}{2}\ge\dfrac{r}{2R},
\]
Combing the above three inequalities, we obtain
\[
\nm{p}{L^\infty(S(K))}\le\nm{p|_{\mc{I}(K)}}{\ell_\infty}
+\dfrac{2m^2Rh_K}{r^2}\nm{p}{L^\infty(S(K))}.
\]
Using the condition on $r$, we obtain~\eqref{eq:uppbd}.  \qed
\begin{remark}\label{remark:convex}
If {\bf Assumption A} is true, then both assumptions in Lemma~\ref{lema:perturb} are valid with $\al=1+\eps$ and $\beta=2m^2/(r\sin\theta)$, respectively.
\end{remark}

In view of the above estimate for $\Lam(m,\mc{I}(K))$, it seems we
should make $r$ as bigger as possible. Hence we should ask for the
largest disk contained in $S(K)$. If $S(K)$ is star shaped to certain
point $x_0$, then $r$ equals to the smallest distance from $x_0$ to
the boundary of $S(K)$ because $S(K)$ is a polygon.

It remains to find an upper bound for $\#\mc{I}(K)$. Under the assumptions on the triangulation $\MTh$ and the assumption on the element patch $S(K)$, it is clear to find an upper bound for $\#\mc{I}(K)$.
\begin{lemma}\label{lema:number}
If {\bf Assumption A1} and {\bf Assumption A2} on the triangulation $\MTh$ and {\bf Assumption A} on the element patch $S(K)$ are valid, then we have
\begin{equation}\label{eq:number}
\#\mc{I}(K)\le \dfrac{\sigma^2\rho_1^2}{N}\dfrac{R^2}{h_K^2}.
\end{equation}
\end{lemma}

\begin{proof}
For any element $K\in\MTh$, using {\bf Assumption A}, we obtain
\[
\#\mc{I}(K)\abs{K}\le\pi R^2.
\]
By {\bf Assumption A1}, we bound $\abs{K}$ from below as
\[
\abs{K}\ge N\sum_{T\in\wt{\MTh}|_K}\abs{T}\ge N\pi\rho_T^2.
\]
Using {\bf Assumption A1} and the consequence {\bf M1}, we have
\[
h_K\le\sigma\rho_1\rho_T.
\]
A combination of the above three inequalities gives~\eqref{eq:number}.
\end{proof}

The upper bound~\eqref{eq:number} is independent of the construction approach of the element patch. For the two approaches based on Moore neighbor and von Neumann neighbor, we have $R\simeq th_K$ with $t$ the recursion depth. Hence we have $\#\mc{I}(K)\simeq t^2$, which is consistent with the upper bound proved in~\cite[Lemma 3.4]{li2012efficient}, in which we have assumed that $S(K)$ is convex and the mesh is quasi-uniform.
\section{IPDG with Reconstructed Space for Poisson Problem}\label{sec:weakform}
We shall use DG method with the reconstructed finite element space to solve
second-order elliptic problem. For the sake of clarity and simplicity, we only consider the Poisson problem
\begin{equation}\label{eq:possion}
-\triangle u=f\quad \text{\;in\;}\Om,\qquad u=0\quad
\text{\;on\;}\pa\Om,
\end{equation}
where $\Om$ is a convex polygonal domain and $f$ is a given function in $L^2(\Om)$. The extension to the general second order elliptic problem is straightforward; see, for example, the numerical examples in the next part. We also mention~\cite{Sun:2017} for the implementation of this reconstructed finite element space together with the DG variational formulation in~\cite{mozolevski2003priori} to biharmonic problem.

The approximating problem is to look for $u_h\in U_h$ such that
\begin{equation}\label{eq:vara}
a_h(\mc{R}u_h,\mc{R}v)=(f,\mc{R}v)_h\quad\text{for all\quad}v\in U_h.
\end{equation}
There are many different DG formulations for this problem as
in~\cite{arnold2002unified} be specifying the bilinear form $a_h$ and the source term $(f,\mc{R}v)_h$. To fix ideas, we focus on the IPDG method in~\cite{arnold1982interior}, where $a_h$ and $(f,\mc{R}v)_h$ are defined for any
$v,w\in V_h$ as
\[
a_h(v,w){:}=\sum_{K\in\MTh}\int_K\na v\cdot\na
w\dx-\sum_{e\in\mc{E}_h}\int_e\Lr{\jump{\na v}\aver{w}+\jump{\na
    w}\aver{v}}\ds
+\sum_{e\in\mc{E}_h}\int_e\dfrac{\eta_e}{h_e}\jump{v}\cdot\jump{w}\ds,
\]
and
\[
(f,\mc{R}v)_h{:}=\sum_{K\in\MTh}\int_Kf(x)\mc{R}v(x)\dx,
\]
where $\eta_e$ is a piecewise positive constant. Here $\mc{E}_h$ is the collection of all edges of $\MTh$, and
$\mc{E}_h^{o}$ is the collection of all the interior edges and
$\mc{E}_h^{\pa}$ is the collection of all boundary edges. Moreover, the average
$\aver{v}$ and the jump $\lj v \rj$ of $v$ is defined as follows. Let
$e$ be a common edge shared by elements $K_1$ and $K_2$, and let $n_1$
and $n_2$ be the outward unit normal at $e$ of $K_1$ and $K_2$,
respectively. Given $v_i{:}=\left. v \right|_{\partial K_i}$, we
define
\[
\aver{v}=\dfrac{1}{2}(v_1 + v_2),\quad \lj v \rj = v_1 n_1 + v_2
n_2,\quad\text{on }\ e\in\mc{E}_h^o.
\]
For a vector-valued function $\varphi$, we define $\varphi_1$ and
$\varphi_2$ analogously and let
\[
\aver{\varphi}=\dfrac{1}{2}(\varphi_1+\varphi_2),\quad \lj \varphi \rj
=\varphi_1\cdot n_1+\varphi_2\cdot n_2,\quad\text{on}\ e\in\mc{E}_h^o.
\]
For $e \in\mc{E}_h^{\pa}$, we set
\[
\lj v \rj = vn,\quad \aver{\varphi}=\varphi.
\]

We define the DG-energy norm for any $v\in V_h$ as
\begin{equation}\label{eq:seminorms}
\enernm{v}=\Lr{\sum_{K\in \MTh}\nm{\na v}{L^2(K)}^{2}+ \sum_{e\in
  \mc{E}_h}\abs{e}^{-1}\nm{\jump{v}}{L^2(e)}^2}^{1/2}.
\end{equation}

Using the {\em Agmon inequality}~\eqref{eq:agmon}, the interpolation
estimates~\eqref{eq:l2app} and~\eqref{eq:h1app}, we obtain, for $g\in
H^{m+1}(\Om)$, there exists $C$ that depends on $N,\sigma,\gamma$ and $m$
such that
\begin{equation}\label{eq:optapp}
\enernm{g-\mc{R}g}\le C(h^m+\Lam_md^m)\snm{g}{H^{m+1}(\Om)},
\end{equation}
which implies that the nearly optimal approximation property of the reconstruction is also valid for
the DG-energy norm.

By definition, we obtain the consistency of $a_h$ in the sense that
\[
a_h(u,\mc{R}v)=(f,\mc{R}v)_h\quad\text{for all\quad}v\in U_h.
\]
Therefore, the Galerkin orthogonality holds true.
\begin{equation}\label{eq:GalerkinOrth}
a_h(u-\mc{R}u_h,\mc{R}v)=0\quad\text{for all\quad}v\in U_h.
\end{equation}
This is the starting point of the error estimate.

By {\em the discrete trace inequality}~\eqref{eq:traceinv}, for sufficiently large $\eta_e$, there exist $\alpha$ and $\beta$ that depend on $N,\sigma,\gamma$ and $m$ such that
\begin{align*}
a_h(\mc{R}v,\mc{R}v)&\ge\al\enernm{\mc{R}v}^2\qquad\text{for all\quad} v\in
U_h,\\
\abs{a_h(\mc{R}v,\mc{R}w)}&\le
\beta\enernm{\mc{R}v}\enernm{\mc{R}w}\qquad\text{for all\quad} v,w\in U_h.
\end{align*}
This immediately gives the well-posedness of the approximation
problem~\eqref{eq:vara}. The error estimate is included in the
following
\begin{theorem}\label{thm:error}
Let $u$ and $u_h$ be the solutions of~\eqref{eq:possion} and~\eqref{eq:vara}, respectively. If {\bf Assumption B} holds, then
\begin{equation}\label{eq:cea}
\enernm{u-\mc{R}u_h}\le\Lr{1+\beta/\al}\enernm{u-\mc{R}u}.
\end{equation}
And if $u\in H^{m+1}(\Om)$, then there exists $C$ that depends on $N,\sigma,\gamma$ and $m$ such that
\begin{equation}\label{eq:enererr}
\enernm{u-\mc{R}u_h}\le C\Lr{h^m+\Lam_md^m}\snm{u}{H^{m+1}(\Om)},
\end{equation}
and
\begin{equation}\label{eq:l2err}
\nm{u-\mc{R}u_h}{L^2(\Om)}\le
C\Lr{h^m+\Lam_md^m}(h+d)\snm{u}{H^{m+1}(\Om)}.
\end{equation}
\end{theorem}
\begin{remark}
If {\bf Assumption A} is valid, then we may reshape the above two estimates into
\begin{equation}\label{eq:finalerr1}
\nm{u-\mc{R}u_h}{L^2(\Om)}+h\enernm{u-\mc{R}u_h}\le Ch^{m+1}\snm{u}{H^{m+1}(\Om)},
\end{equation}
where $C$ depends on $N,\sigma,\gamma,m$ and the recursion depth $t$ of the element patch.

If $S(K)$ is convex, the above error estimate~\eqref{eq:finalerr1} remains true, while
$C$ depends on $N,\sigma,m$ and $t$ but is independent of the chunkness parameter $\gamma$.
\end{remark}

\begin{proof}
Denote $v=\mc{R}u-\mc{R}u_h$, we obtain
\[
a_h(v,v)=a_h(\mc{R}u-u,v)+a_h(u-\mc{R}u_h,v)=a_h(\mc{R}u-u,v),
\]
where we have used the Galerkin orthogonality~\eqref{eq:GalerkinOrth}
in the last step. This implies
\[
\enernm{\mc{R}u-\mc{R}u_h}\le\dfrac{\beta}{\al}\enernm{u-\mc{R}u},
\]
which together with the triangle inequality implies~\eqref{eq:cea}.

Substituting the interpolate estimate~\eqref{eq:optapp}
into~\eqref{eq:cea}, we obtain~\eqref{eq:enererr}.

To show the L$^2$-error estimate~\eqref{eq:l2err}, we use the standard duality argument. Let $\phi$ be the solution of
\[
-\triangle\phi=u-\mc{R}u_h\quad \text{in\;}\Om\qquad \phi=0\quad\text{on\;}\pa\Om.
\]
Using an integration by parts, using the Galerkin orthogonality~\eqref{eq:GalerkinOrth} andthe interpolation estimate~\eqref{eq:optapp} with $m=1$, we obtain
\begin{align*}
\nm{u-\mc{R}u_h}{L^2(\Om)}^2&=\int_{\Om}-\triangle\phi(u-\mc{R}u_h)\dx=a_h(u-\mc{R}u_h,\phi)
=a_h(u-\mc{R}u_h,\phi-\mc{R}\phi)\\
&\le\beta\enernm{u-\mc{R}u_h}\enernm{\phi-\mc{R}\phi}\\
&\le C(h+d)\enernm{u-\mc{R}u_h}\snm{\phi}{H^2(\Om)}.
\end{align*}
Next, as $\Om$ is convex, elliptic regularity gives \(\snm{\phi}{H^2(\Om)}\le C_r\nm{u-\mc{R}u_h}{L^2(\Om)}\) with $C_r$ depending only on the domain $\Om$. Hence, using the energy estimate~\eqref{eq:enererr}, we obtain the L$^2$-error estimate~\eqref{eq:l2err} and complete the proof.
\end{proof}
\section{Numerical examples}\label{sec:examples}
In this section, we present some numerical examples for general second
order elliptic problem of the form
\begin{equation}\label{eq:ellgen}
-\div\Lr{A(x)\na u(x)}=f(x)
\end{equation}
supplemented with various boundary conditions. Here $A$ is a two by
two matrix that satisfies
\begin{equation}\label{eq:ellcons}
c_1(\xi_1^2+\xi_2^2)\le\sum_{i,j=1}^2A_{ij}(x)\xi_i\xi_j\le
c_2\Lr{\xi_1^2+\xi_2^2},\quad\text{a.e.\;}x\in\Omega,
\end{equation}
where $0<c_1\le c_2$ and $\xi_1,\xi_2\in\mb{R}$.

In all the examples below, we take the penalty term $\eta_e$ large
enough to guarantee the coercivity of $a_h$. To be more precise, we
let $\eta_e\geq 3c_2$ for the interior edges $e\in\mc{E}_h^{o}$, where $c_2$ is the ellipticity constant in~\eqref{eq:ellcons}; and $\eta_e$ is taken as $k m^2$ for boundary
edge $e\in\mc{E}_h^{\partial}$ with $k$ a positive constant, while $k$ may vary for different
examples. A direct solver is employed to solve all the resulting
linear systems.
\vskip .5cm \textbf{Example 1.} In the first example, we consider a 2D
Laplace equation with homogeneous Neumann boundary condition posed on
the unit square, i.e., $A(x)$ is a $2\x 2$ identity matrix. We
assume an exact solution and a smooth source term $f$ as
\[
  u(x,y)=\sin(2\pi x)\sin(2 \pi y),\quad f=8\pi^2\sin(2\pi x)\sin(2\pi
  y).
\]
We consider quasi-uniform triangular and quadrilateral meshes, which
are generated by the software {\em Gmsh}~\cite{geuzaine2009gmsh}, as
shown in Figure~\ref{tri_mesh and quad mesh}.
\begin{figure}
  \begin{center}
    \includegraphics[width=0.4\textwidth]{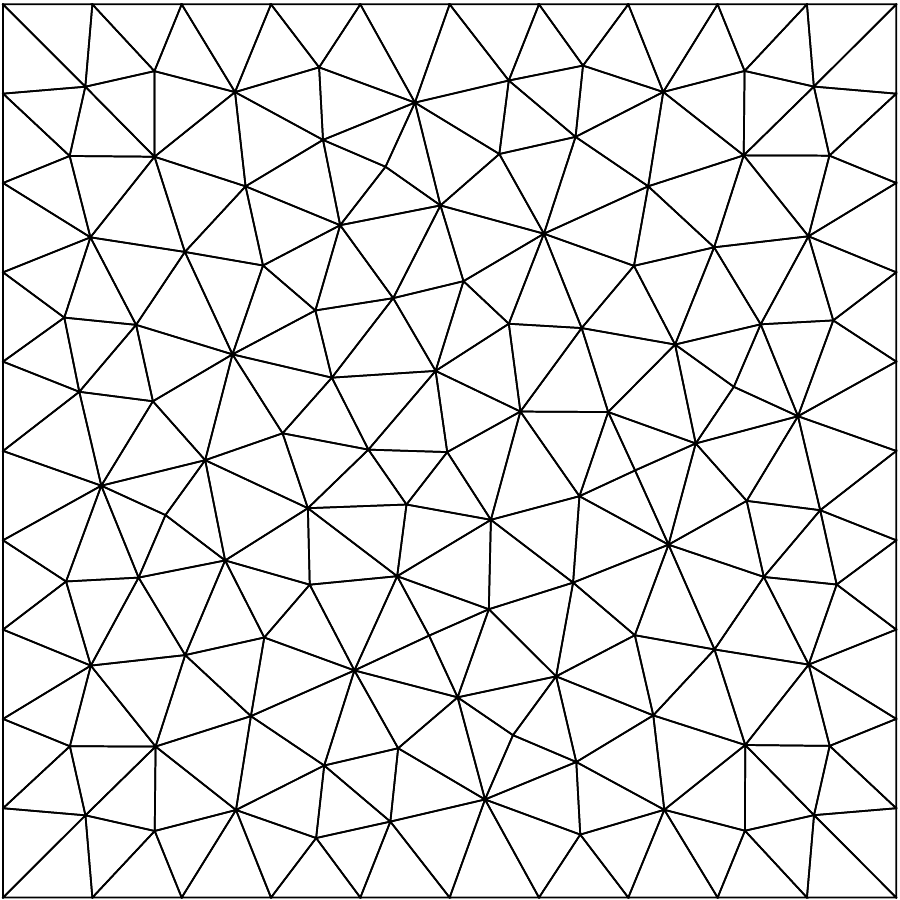}
    \hspace{1cm} \includegraphics[width=0.4\textwidth]{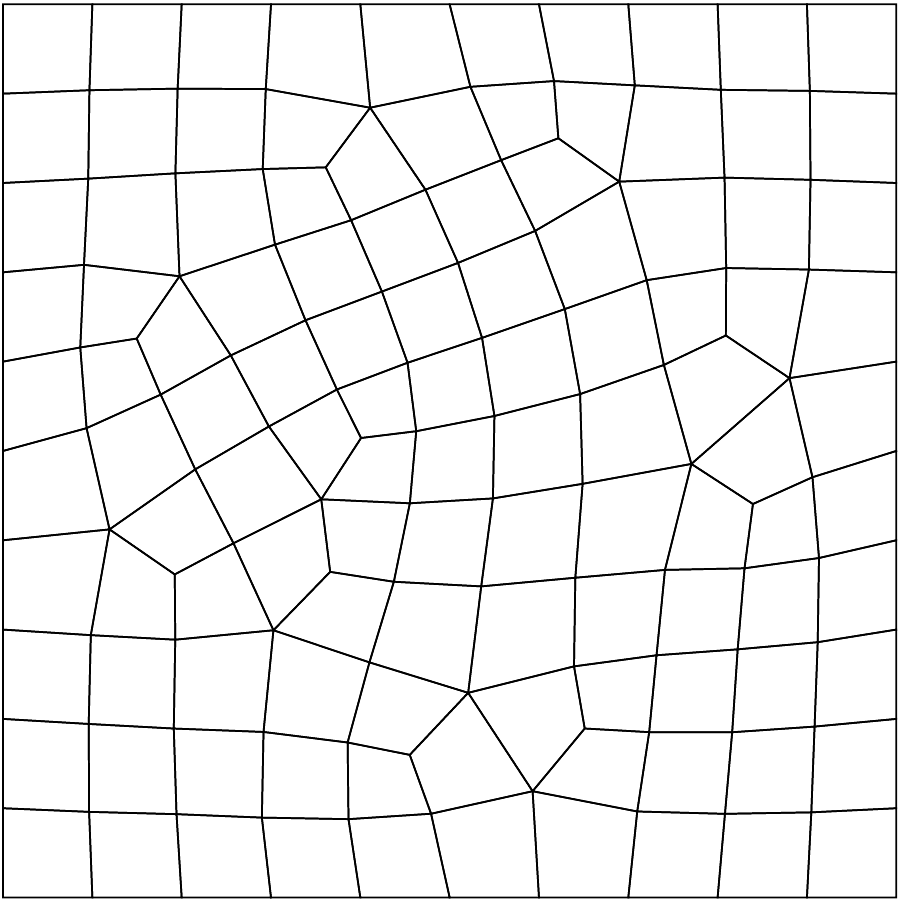}
    \caption{The triangular and quadrilateral meshes for example 1.}
  \label{tri_mesh and quad mesh}
  \end{center}
\end{figure}

We plot convergence rate in Figure~\ref{tri_error} and
Figure~\ref{quad_error} for the triangular and quadrilateral meshes,
respectively. It is clear that the method converges in the
energy norm with rate $m$ and converges in $L^2$ norm with rate $m+1$, where $m$ is the reconstruction order, which is consistent with the
theoretical prediction in Theorem~\ref{thm:error}. The numerical
errors and convergence rates are also presented in Table~\ref{tri_mesh} and
Table~\ref{quad_mesh} for the triangular and quadrilateral meshes,
respectively.
\begin{figure}
  \begin{center}
    \includegraphics[width=0.48\textwidth]{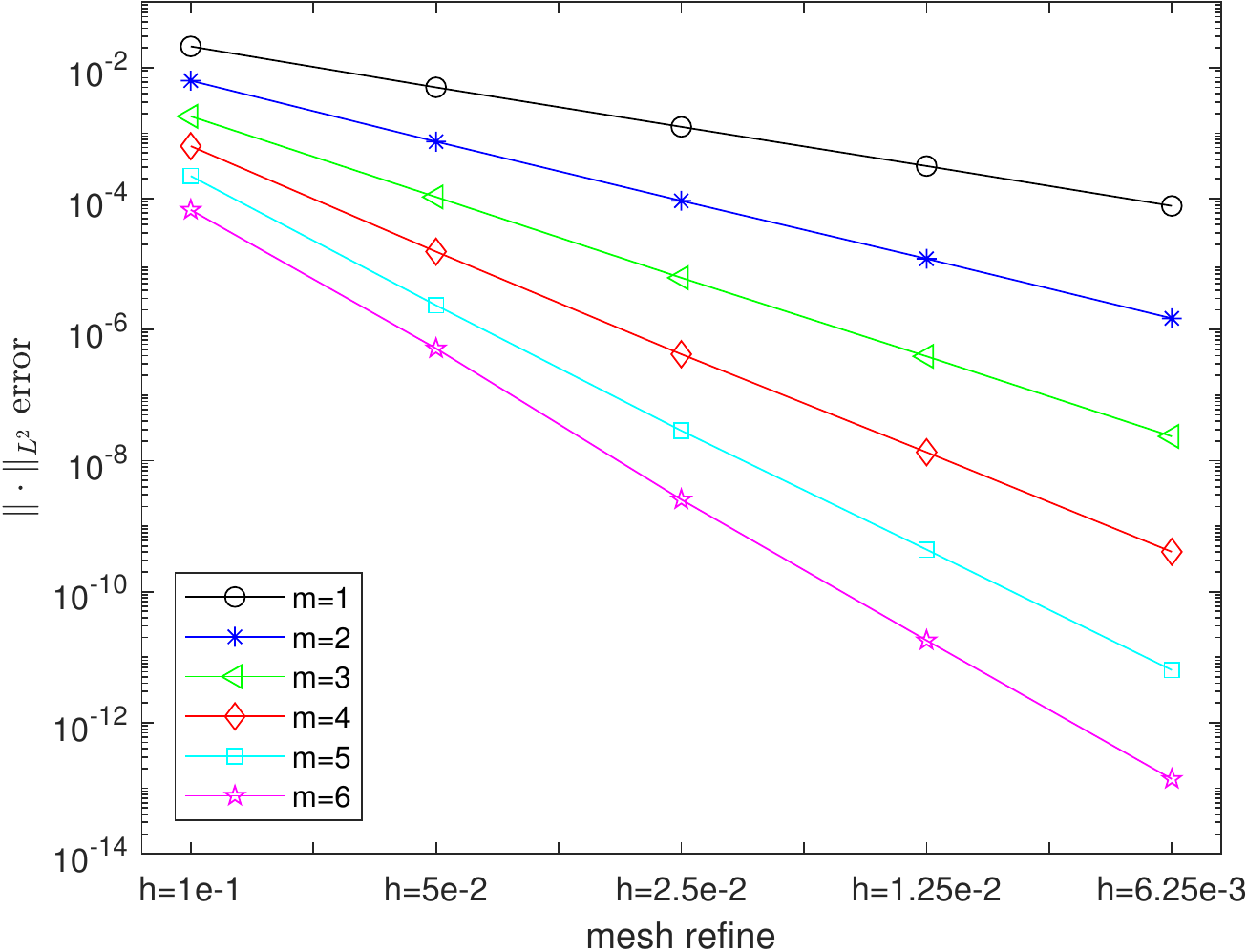}
    \includegraphics[width=0.48\textwidth]{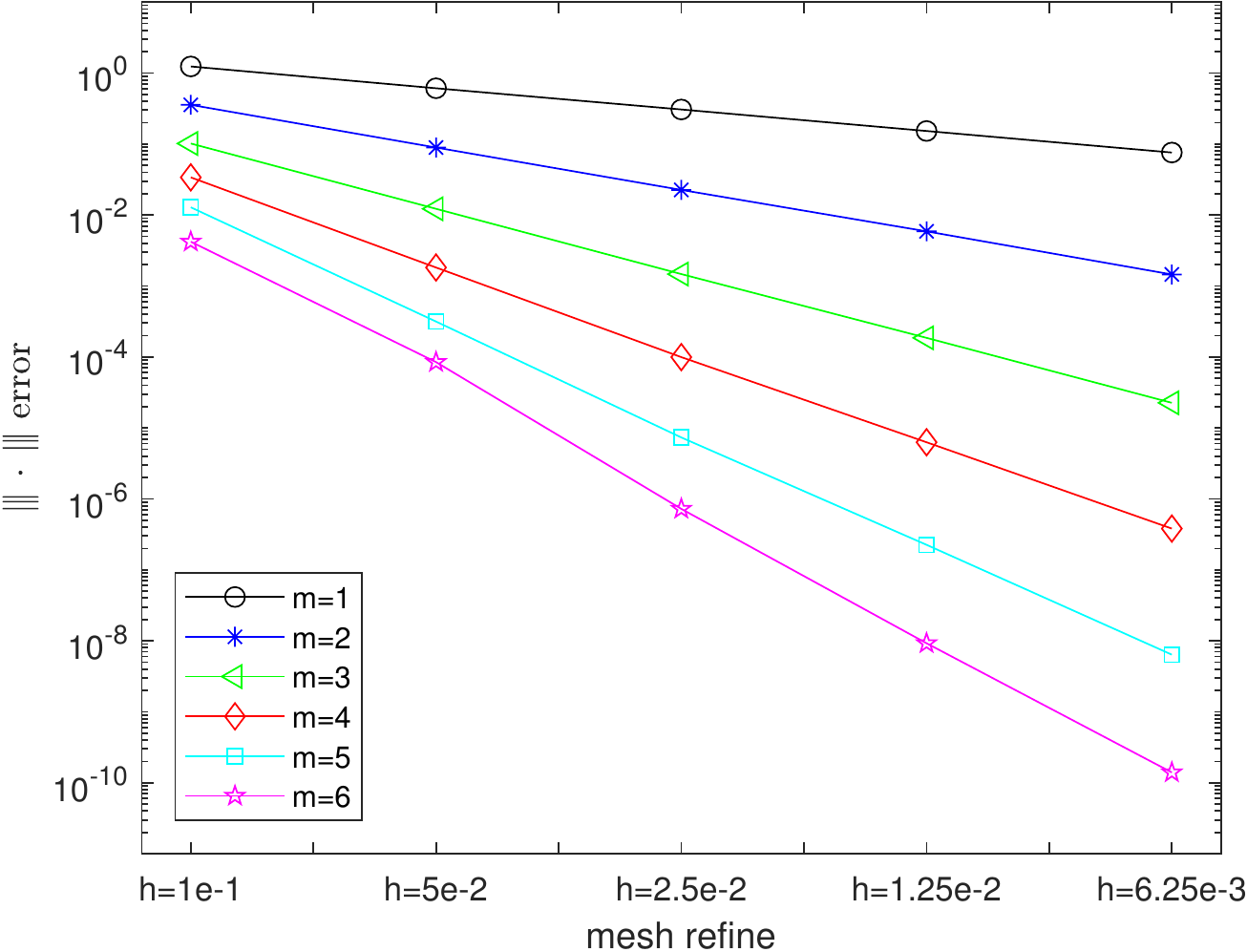}
    \caption{The convergence rate in $L^2$ norm (left) and the energy
      norm (right) for different reconstruction order $m$ on triangular
      meshes for Example 1.}\label{tri_error}
  \end{center}
\end{figure}
\begin{table}[]
  \begin{center}
  \scalebox{1.0}{
    \begin{tabular}{|c|c|c|c|c|c|c|c|}
      \hline & & h=1.0e-1 & 5.0e-2 & 2.5e-2 & 1.25e-2 & 6.25e-3 &
      \\ \cline{3-7} \multirow{-2}{*}{$m$} & \multirow{-2}{*}{Norms} &
      error & error & error & error & error & \multirow{-2}{*}{Rate}
      \\ \hline

      & $\nm{u-u_h}{L^2}$ & 2.10e-02 & 4.98e-03 & 1.24e-03 & 3.14e-04
      & 7.73e-05 & 2.02 \\ \cline{2-8} \multirow{-2}{*}{1} &
      \multicolumn{1}{c|}{$\enernm{u-u_h}$} &
      \multicolumn{1}{c|}{1.23e+00} & 6.10e-01 & 3.06e-01 & 1.52e-01 &
      7.58e-02 & 1.01 \\ \hline

      & $\nm{u-u_h}{L^2}$ & 6.32e-03 & 7.40e-04 & 9.26e-05 & 1.20e-05
      & 1.48e-06 & 3.00 \\ \cline{2-8} \multirow{-2}{*}{2} &
      \multicolumn{1}{c|}{$\enernm{u-u_h}$} &
      \multicolumn{1}{c|}{3.56e-01} & 8.91e-02 & 2.25e-02 & 5.87e-03 &
      1.45e-03 & 1.98 \\ \hline

      & $\nm{u-u_h}{L^2}$ & 1.80e-03 & 1.05e-04 & 6.18e-06 & 3.91e-07
      & 2.34e-08 & 4.05 \\ \cline{2-8} \multirow{-2}{*}{3} &
      \multicolumn{1}{c|}{$\enernm{u-u_h}$} &
      \multicolumn{1}{c|}{1.01e-01} & 1.22e-02 & 1.47e-03 & 1.85e-04 &
      2.25e-05 & 3.03 \\ \hline

      & $\nm{u-u_h}{L^2}$ & 6.32e-04 & 1.54e-05 & 4.21e-07 & 1.34e-08
      & 4.05e-10 & 5.13 \\ \cline{2-8} \multirow{-2}{*}{4} &
      \multicolumn{1}{c|}{$\enernm{u-u_h}$} &
      \multicolumn{1}{c|}{3.38e-02} & 1.81e-03 & 9.93e-05 & 6.27e-06 &
      3.81e-07 & 4.10 \\ \hline

      & $\nm{u-u_h}{L^2}$ & 2.21e-04 & 2.35e-06 & 2.86e-08 & 4.39e-10
      & 6.36e-12 & 6.25 \\ \cline{2-8} \multirow{-2}{*}{5} &
      \multicolumn{1}{c|}{$\enernm{u-u_h}$} &
      \multicolumn{1}{c|}{1.28e-02} & 3.15e-04 & 7.33e-06 & 2.25e-07 &
      6.38e-09 & 5.23 \\ \hline

      & $\nm{u-u_h}{L^2}$ & 6.73e-05 & 5.16e-07 & 2.54e-09 & 1.80e-11
      & 1.38e-13 & 7.25 \\ \cline{2-8} \multirow{-2}{*}{6} &
      \multicolumn{1}{c|}{$\enernm{u-u_h}$} &
      \multicolumn{1}{c|}{4.20e-03} & 8.45e-05 & 7.22e-07 & 9.23e-09 &
      1.39e-10 & 6.28 \\ \hline
  \end{tabular}}\vspace{.3cm}
  \caption{Errors on the triangular meshes for Example
    1.}\label{tri_mesh}
  \end{center}
\end{table}
\begin{figure}
  \begin{center}
    \includegraphics[width=0.48\textwidth]{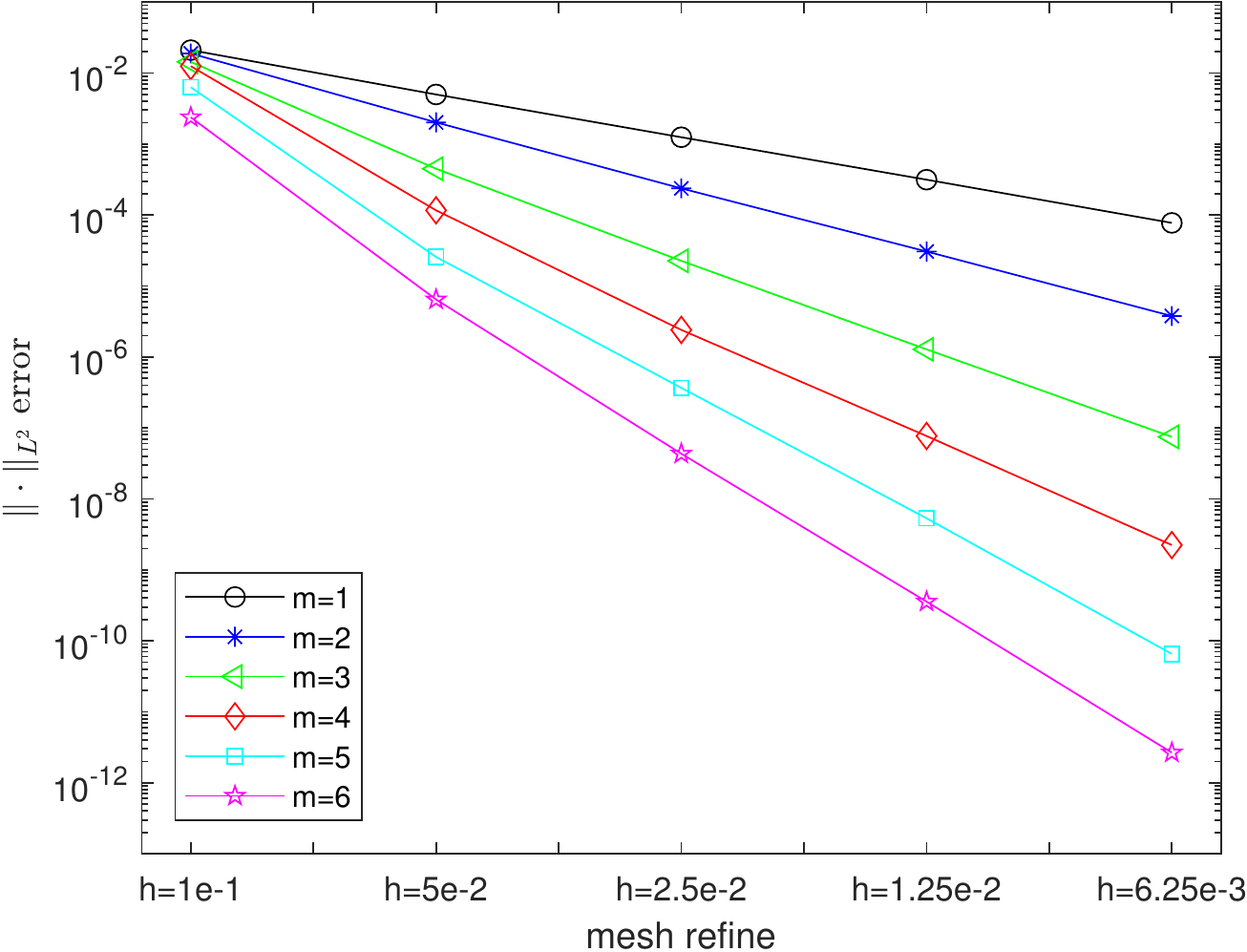}
    \includegraphics[width=0.48\textwidth]{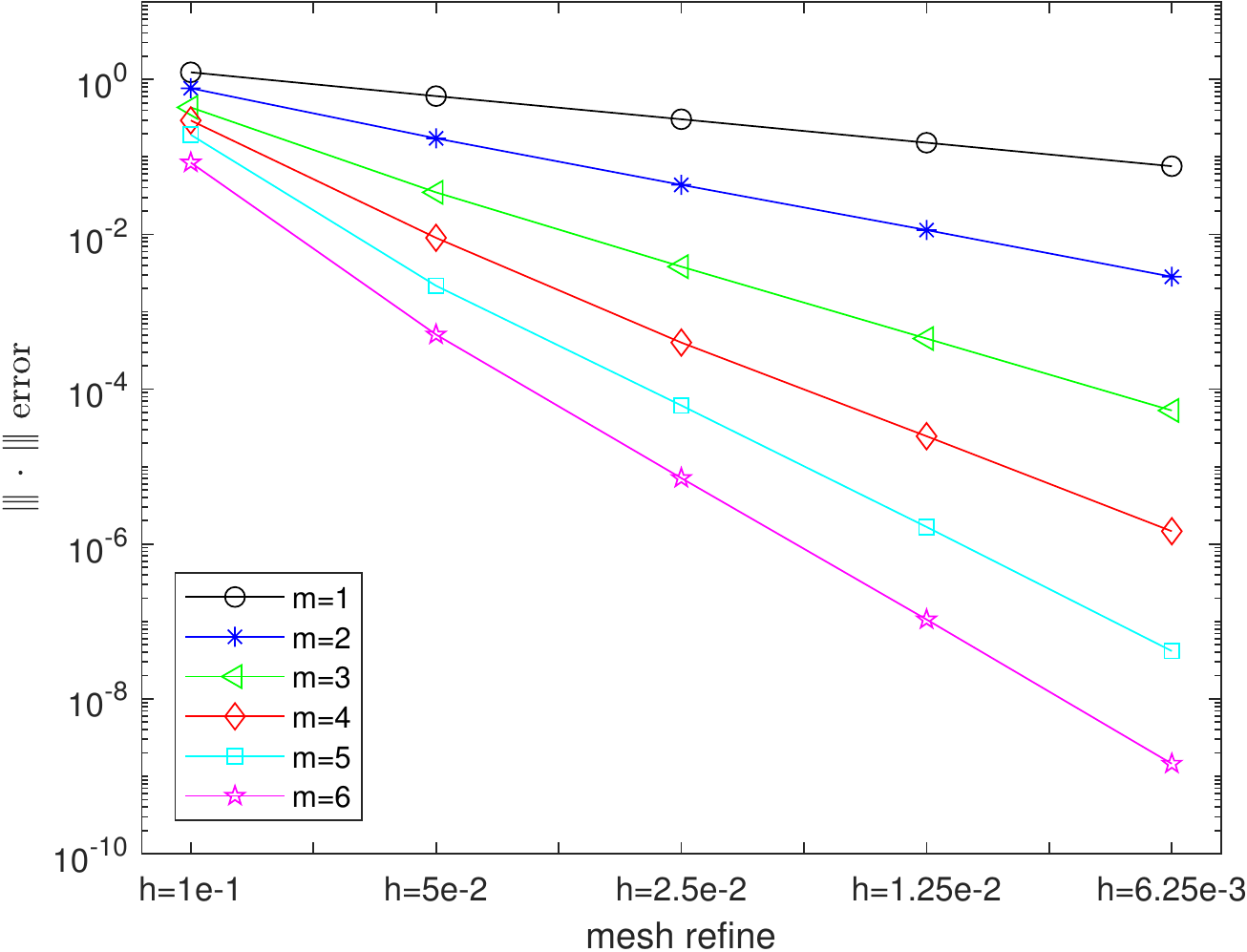}
    \caption{Convergence rate in $L^2$ norm (left) and the energy norm
      (right) for different reconstruction order $m$ on quadrilateral
      meshes for Example 1.}
    \label{quad_error}
  \end{center}
\end{figure}
\begin{table}[]
  \begin{center}
  \scalebox{1.0}{
    \begin{tabular}{|c|c|c|c|c|c|c|c|}
      \hline & & h=1.0e-1 & 5.0e-2 & 2.5e-2 & 1.25e-2 & 6.25e-3 &
      \\ \cline{3-7} \multirow{-2}{*}{$m$} & \multirow{-2}{*}{Norms} &
      error & error & error & error & error & \multirow{-2}{*}{Rate}
      \\ \hline

      & $\nm{u-u_h}{L^2}$ & 3.30e-02 & 7.41e-03 & 1.78e-03 & 4.47e-04
      & 1.06e-04 & 2.01 \\ \cline{2-8} \multirow{-2}{*}{1} &
      \multicolumn{1}{c|}{$\enernm{u-u_h}$} &
      \multicolumn{1}{c|}{1.51e+00} & 6.93e-01 & 3.48e-01 & 1.73e-01 &
      8.44e-02 & 1.01 \\ \hline

      & $\nm{u-u_h}{L^2}$ & 1.88e-02 & 2.01e-03 & 2.37e-04 & 3.05e-05
      & 3.77e-06 & 3.06 \\ \cline{2-8} \multirow{-2}{*}{2} &
      \multicolumn{1}{c|}{$\enernm{u-u_h}$} &
      \multicolumn{1}{c|}{7.70e-01} & 1.73e-01 & 4.35e-02 & 1.13e-02 &
      2.83e-03 & 2.01 \\ \hline

      & $\nm{u-u_h}{L^2}$ & 1.43e-02 & 4.48e-04 & 2.25e-05 & 1.27e-06
      & 7.47e-08 & 4.35 \\ \cline{2-8} \multirow{-2}{*}{3} &
      \multicolumn{1}{c|}{$\enernm{u-u_h}$} &
      \multicolumn{1}{c|}{4.35e-01} & 3.48e-02 & 3.82e-03 & 4.51e-04 &
      5.30e-05 & 3.22 \\ \hline

      & $\nm{u-u_h}{L^2}$ & 1.25e-02 & 1.16e-04 & 2.40e-06 & 7.68e-08
      & 2.23e-09 & 5.53 \\ \cline{2-8} \multirow{-2}{*}{4} &
      \multicolumn{1}{c|}{$\enernm{u-u_h}$} &
      \multicolumn{1}{c|}{2.95e-01} & 9.01e-03 & 4.00e-04 & 2.47e-05 &
      1.46e-06 & 4.37 \\ \hline

      & $\nm{u-u_h}{L^2}$ & 6.31e-03 & 2.56e-05 & 3.66e-07 & 5.32e-09
      & 6.53e-11 & 6.52 \\ \cline{2-8} \multirow{-2}{*}{5} &
      \multicolumn{1}{c|}{$\enernm{u-u_h}$} &
      \multicolumn{1}{c|}{1.93e-01} & 2.16e-03 & 6.16e-05 & 1.66e-06 &
      4.15e-08 & 5.46 \\ \hline

      & $\nm{u-u_h}{L^2}$ & 2.37e-03 & 6.44e-06 & 4.32e-08 & 3.56e-10
      & 2.65e-12 & 7.36 \\ \cline{2-8} \multirow{-2}{*}{6} &
      \multicolumn{1}{c|}{$\enernm{u-u_h}$} &
      \multicolumn{1}{c|}{8.43e-02} & 5.07e-04 & 7.10e-06 & 1.06e-07 &
      1.45e-09 & 6.38 \\ \hline
  \end{tabular}}\vspace{.3cm}
  \caption{Errors on the quadrilateral mesh for Example 1.}
  \label{quad_mesh}
  \end{center}
\end{table}
\vskip .5cm \textbf{Example 2.} This example is taken
from~\cite[Example 4.1]{brezzi2005family}. We consider the Dirichlet
boundary value problem in the unit square $(0,1)^2$ with the exact
solution
\[
  u(x,y)= x^{3} y^{2}+ x \sin(2\pi x y)\sin(2 \pi y).
\]
The coefficient matrix $A$ is taken as
\[
A(x,y)=\begin{pmatrix} (x+1)^2+y^2 & -xy \\ -xy & (x+1)^2
\end{pmatrix}.
\]
The force $f$ is then determined by the equation~\eqref{eq:ellgen}. We
solve this problem over a sequence of hexagonal meshes as shown in
Figure~\ref{hexa_mesh and mixed mesh}, which are generated by a
Voronoi tessellation. The mesh contains elements with different shapes such as hexagons, pentagons, and
quadrilaterals. The complexity of complicate element shape does not bring in
extra difficulties in implementation. The errors and convergence rate
are reported in Table~\ref{hexagon_mesh} and Figure~\ref{hexa_error},
respectively, which are agreed with the theoretical prediction.
\begin{figure}
  \begin{center}
    \includegraphics[width=0.4\textwidth]{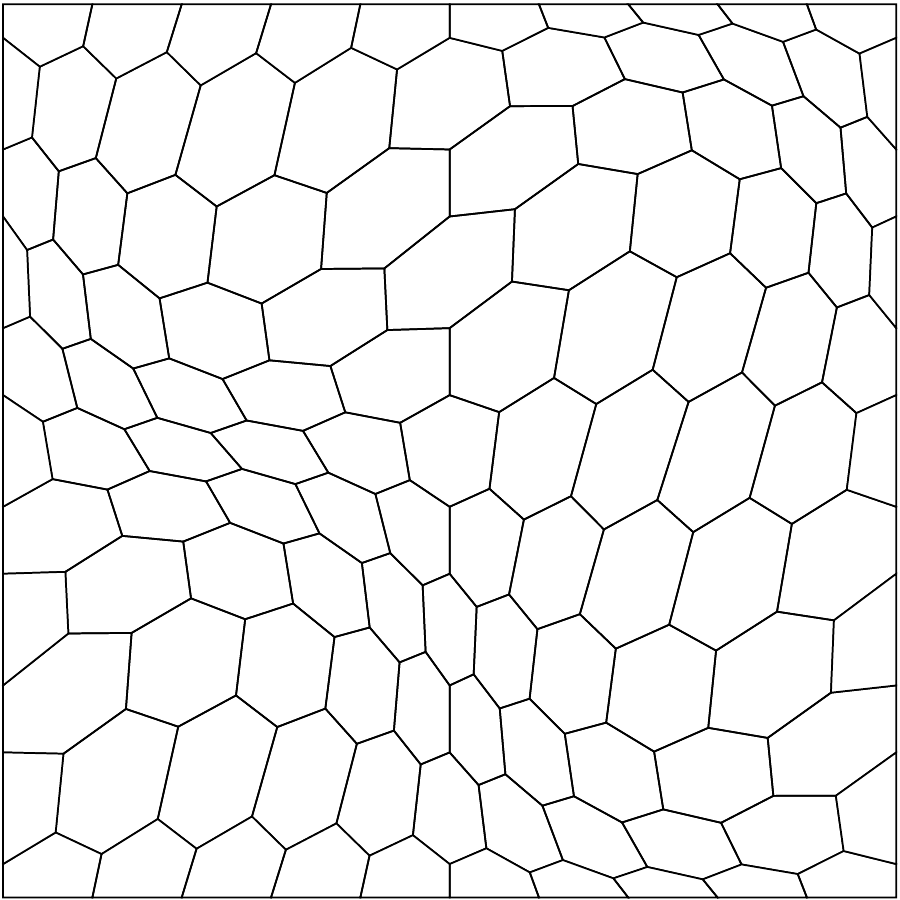}
    \includegraphics[width=0.4\textwidth]{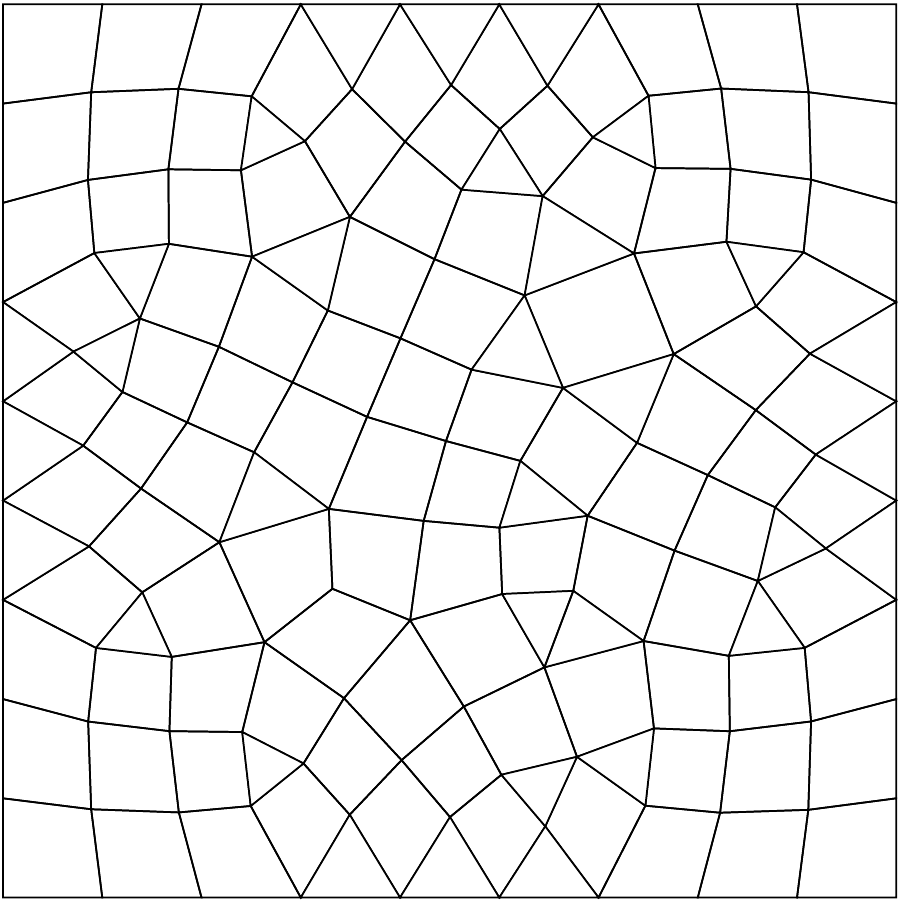}
    \caption{The hexagonal mesh and the mixed mesh for Example 2 and
      Example 3.}
    \label{hexa_mesh and mixed mesh}
  \end{center}
\end{figure}
\begin{figure}
  \begin{center}
    \includegraphics[width=0.48\textwidth]{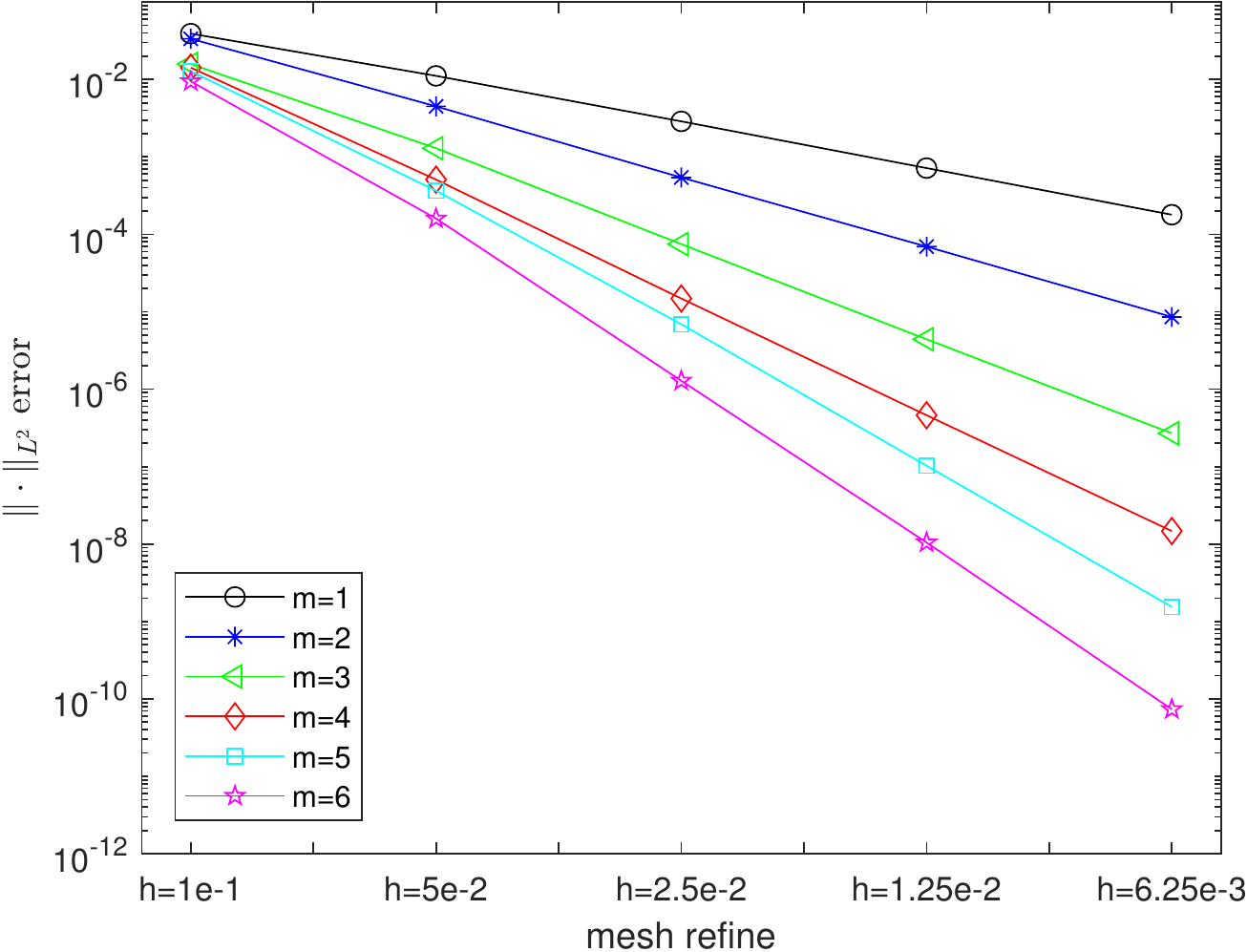}
    \includegraphics[width=0.48\textwidth]{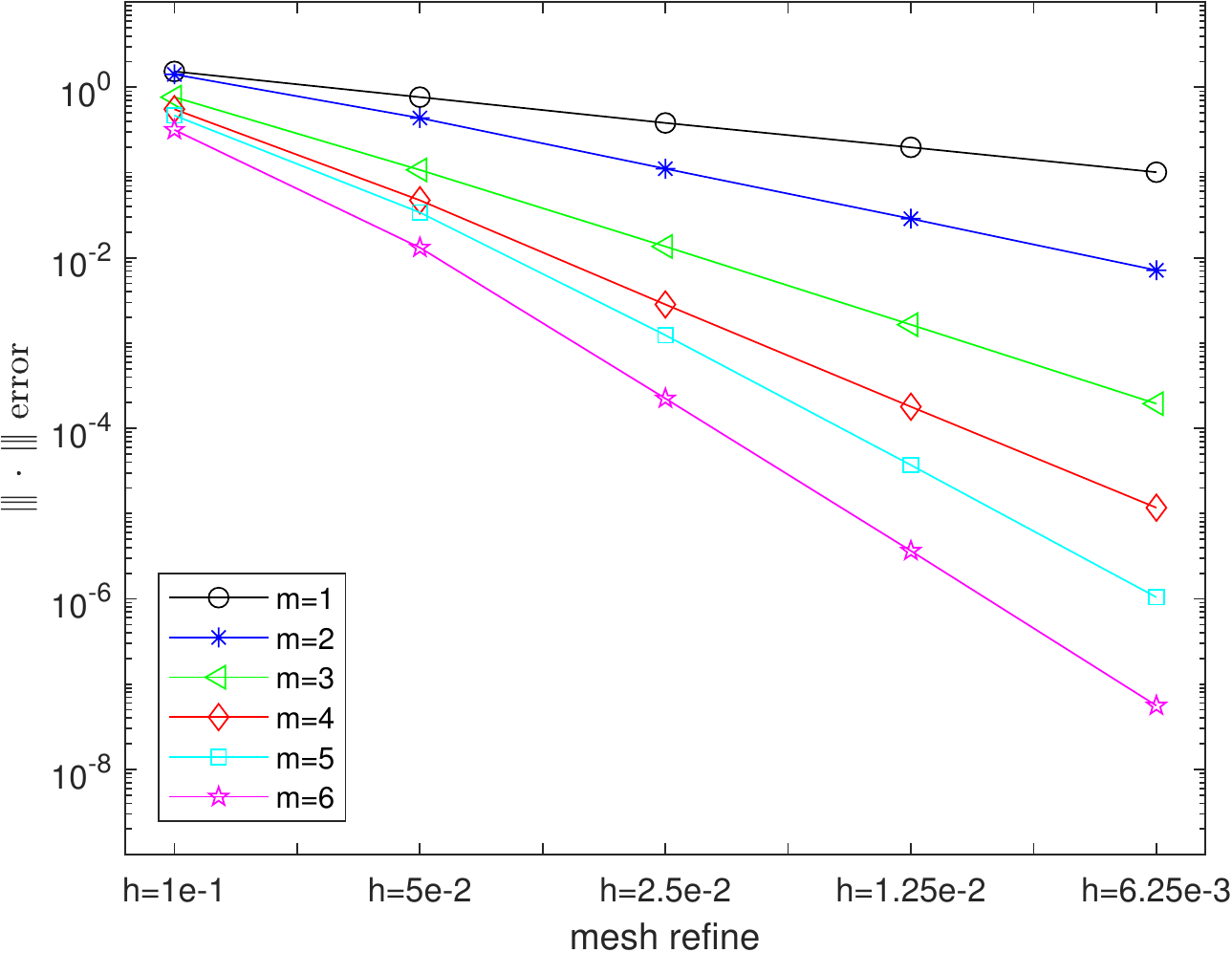}
    \caption{The convergence rate in $L^2$ norm (left) and the energy
      norm (right) for different reconstruction order $m$ over
      hexagonal meshes for Example 2.}\label{hexa_error}
  \end{center}
\end{figure}
\begin{table}[]
  \begin{center}
  \scalebox{1.0}{
    \begin{tabular}{|c|c|c|c|c|c|c|c|}
      \hline & & N=1.15e+2 & 4.30e+2 & 1.66e+3 & 6.52e+3 & 2.58e+4 &
      \\ \cline{3-7} \multirow{-2}{*}{$m$} & \multirow{-2}{*}{Norms} &
      error & error & error & error & error & \multirow{-2}{*}{Rate}
      \\ \hline

      & $\nm{u-u_h}{L^2}$ & 3.91e-02 & 1.11e-02 & 2.87e-03 & 7.17e-04
      & 1.79e-04 & 1.95 \\ \cline{2-8} \multirow{-2}{*}{1} &
      \multicolumn{1}{c|}{$\enernm{u-u_h}$} &
      \multicolumn{1}{c|}{1.53e+00} & 7.67e-01 & 3.82e-01 & 1.97e-01 &
      1.01e-01 & 0.98 \\ \hline

      & $\nm{u-u_h}{L^2}$ & 3.35e-02 & 4.50e-03 & 5.42e-04 & 6.94e-05
      & 8.57e-06 & 2.99 \\ \cline{2-8} \multirow{-2}{*}{2} &
      \multicolumn{1}{c|}{$\enernm{u-u_h}$} &
      \multicolumn{1}{c|}{1.41e+00} & 4.36e-01 & 1.11e-01 & 2.87e-02 &
      7.15e-03 & 1.92 \\ \hline

      & $\nm{u-u_h}{L^2}$ & 1.58e-02 & 1.29e-03 & 7.48e-05 & 4.41e-06
      & 2.69e-07 & 3.99 \\ \cline{2-8} \multirow{-2}{*}{3} &
      \multicolumn{1}{c|}{$\enernm{u-u_h}$} &
      \multicolumn{1}{c|}{7.67e-01} & 1.07e-01 & 1.35e-02 & 1.64e-03 &
      1.94e-04 & 2.99 \\ \hline

      & $\nm{u-u_h}{L^2}$ & 1.42e-02 & 5.09e-04 & 1.48e-05 & 4.61e-07
      & 1.47e-08 & 4.99 \\ \cline{2-8} \multirow{-2}{*}{4} &
      \multicolumn{1}{c|}{$\enernm{u-u_h}$} &
      \multicolumn{1}{c|}{5.53e-01} & 4.73e-02 & 2.83e-03 & 1.79e-04 &
      1.17e-05 & 3.91 \\ \hline

      & $\nm{u-u_h}{L^2}$ & 1.28e-02 & 3.65e-04 & 6.89e-06 & 1.02e-07
      & 1.54e-09 & 5.77 \\ \cline{2-8} \multirow{-2}{*}{5} &
      \multicolumn{1}{c|}{$\enernm{u-u_h}$} &
      \multicolumn{1}{c|}{4.68e-01} & 3.40e-02 & 1.23e-03 & 3.71e-05 &
      1.04e-06 & 4.74 \\ \hline

      & $\nm{u-u_h}{L^2}$ & 9.45e-03 & 1.59e-04 & 1.28e-06 & 1.04e-08
      & 7.33e-11 & 6.78 \\ \cline{2-8} \multirow{-2}{*}{6} &
      \multicolumn{1}{c|}{$\enernm{u-u_h}$} &
      \multicolumn{1}{c|}{3.17e-01} & 1.31e-02 & 2.23e-04 & 3.65e-06 &
      5.58e-08 & 5.67\\ \hline
  \end{tabular}}\vspace{.3cm}
  \caption{Errors on the hexagonal meshes for Examples 2. $N$ is the
    number of the total degrees of freedom.}\label{hexagon_mesh}
  \end{center}
\end{table}
\vskip .5cm \textbf{Example 3.} We consider a Neumann boundary value
problem in the unit square with the exact solution
\[
  u(x,y)=\exp\left(\frac{x^2+y^2}{2}\right)+\sin(2\pi (x+y))\sin(2\pi
  y),
\]
and we take the coefficients matrix as
\[
A(x,y)=\begin{pmatrix} 3+\cos(2\pi x) & x-y \\ x-y & 3-\sin(2\pi y)
\end{pmatrix}.
\]

The meshes are generated by {\em Gmsh}~\cite{geuzaine2009gmsh} again,
which contains both triangles and quadrilaterals as shown in
Figure~\ref{hexa_mesh and mixed mesh}. In this example, the sampling
points are randomly selected inside the element instead of the element
barycenters. We perturb each barycenter with a uniform distribution
random vector $\xi\in\mb{R}^2$ with $\abs{\xi}=0.1h_K$, which
guarantees the perturbed sampling points are still located in the
interior of the corresponding element. We show in
Figure~\ref{patch_ponits} an example of the perturbed sampling points.

Errors are given in Table~\ref{mixed_mesh} and
Figure~\ref{mixed_error}. Again we achieved the same convergence rates
as the previous examples. This indicates that the method is robust
with respect to the perturbation of the sampling points as shown in
Lemma~\ref{theorem:localapp}.
\begin{figure}
  \begin{center}
    \includegraphics[width=0.48\textwidth]{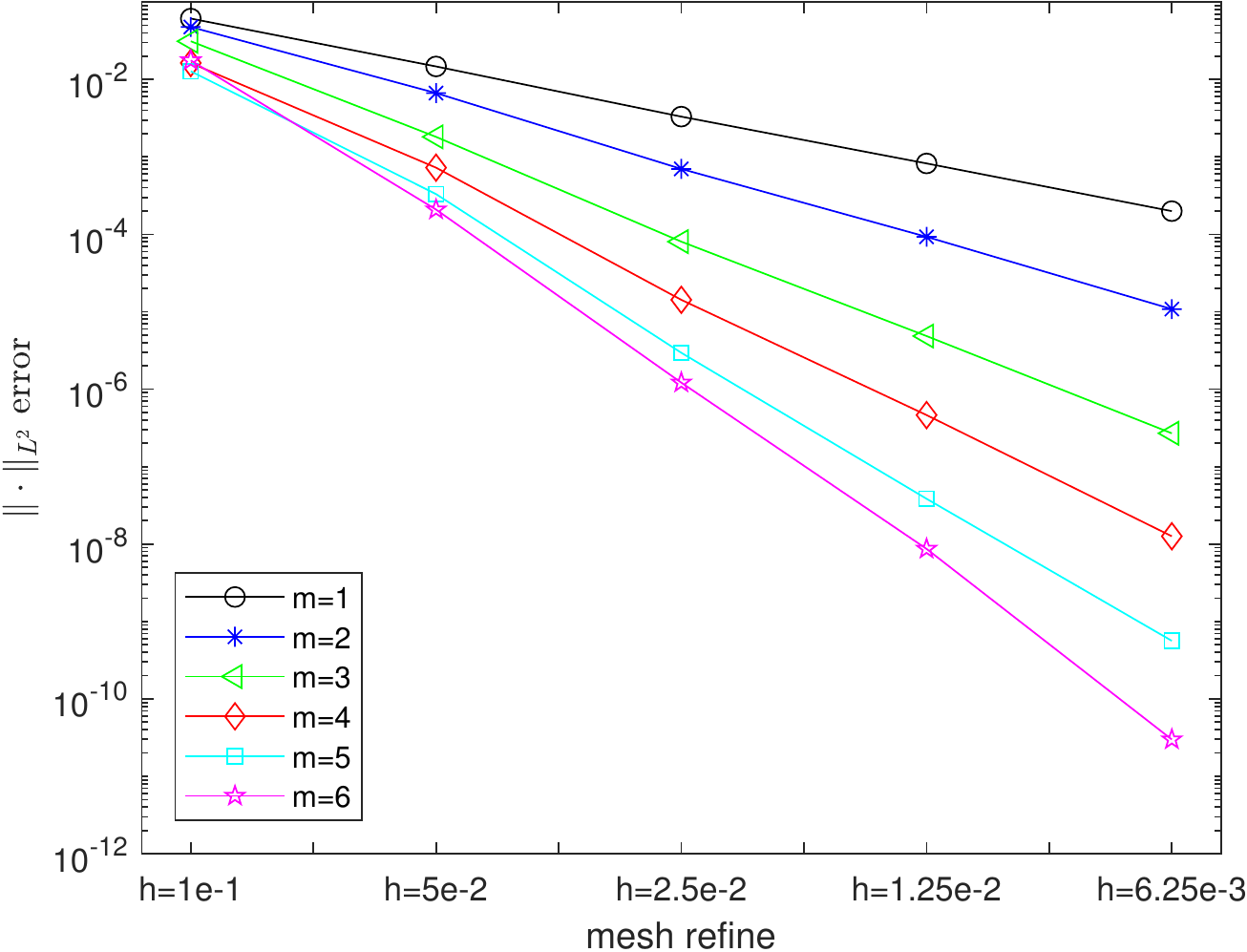}
    \includegraphics[width=0.48\textwidth]{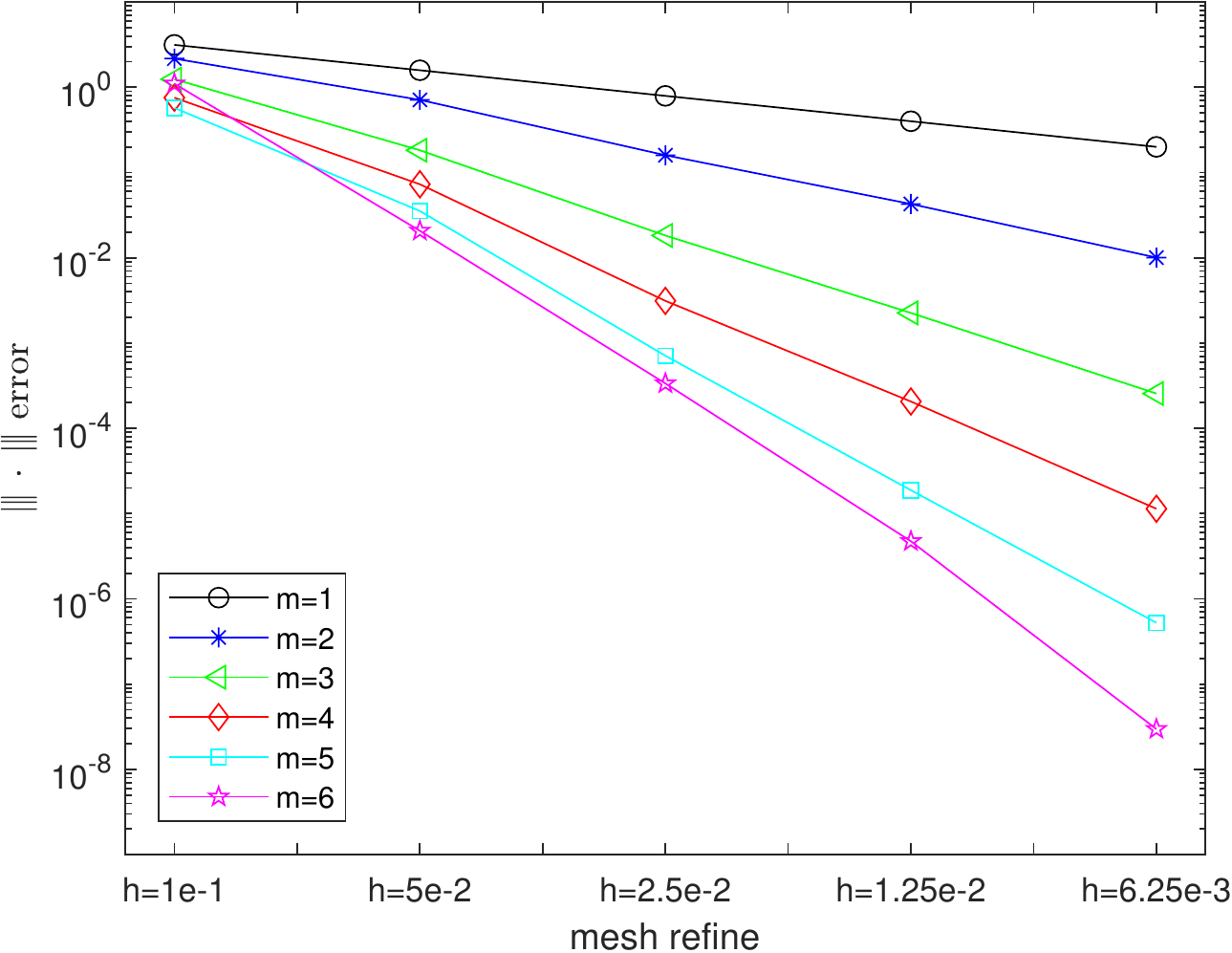}
    \caption{The convergence rate in $L^2$ norm (left) and the energy
      norm (right) for different $m$ for Example 3.}
  \label{mixed_error}
  \end{center}
\end{figure}
\begin{table}[]
  \begin{center}
  \scalebox{1.0}{
    \begin{tabular}{|c|c|c|c|c|c|c|c|}
      \hline & & N=1.29e+2 & 5.10e+2 & 2.33e+3 & 9.24e+3 & 3.80e+4 &
      \\ \cline{3-7} \multirow{-2}{*}{$m$} & \multirow{-2}{*}{Norms} &
      error & error & error & error & error & \multirow{-2}{*}{Rate}
      \\ \hline

      & $\nm{u-u_h}{L^2}$ & 6.66e-02 & 1.48e-02 & 3.47e-03 & 8.27e-04
      & 2.01e-04 & 2.07 \\ \cline{2-8} \multirow{-2}{*}{1} &
      \multicolumn{1}{c|}{$\enernm{u-u_h}$} &
      \multicolumn{1}{c|}{3.35e+00} & 1.57e+00 & 8.44e-01 & 4.04e-01 &
      2.02e-01 & 0.99\\ \hline

      & $\nm{u-u_h}{L^2}$ & 4.76e-02 & 6.68e-03 & 7.00e-04 & 9.29e-05
      & 1.08e-05 & 3.04 \\ \cline{2-8} \multirow{-2}{*}{2} &
      \multicolumn{1}{c|}{$\enernm{u-u_h}$} &
      \multicolumn{1}{c|}{2.17e+00} & 7.12e-01 & 1.59e-01 & 4.28e-02 &
      1.00e-02 & 1.96 \\ \hline

      & $\nm{u-u_h}{L^2}$ & 3.11e-02 & 1.81e-03 & 8.04e-05 & 4.85e-06
      & 2.69e-07 & 4.22 \\ \cline{2-8} \multirow{-2}{*}{3} &
      \multicolumn{1}{c|}{$\enernm{u-u_h}$} &
      \multicolumn{1}{c|}{1.24e+00} & 1.82e-01 & 1.82e-02 & 2.25e-03 &
      2.56e-04 & 3.08\\ \hline

      & $\nm{u-u_h}{L^2}$ & 1.63e-02 & 7.26e-04 & 1.42e-05 & 4.63e-07
      & 1.26e-08 & 5.12 \\ \cline{2-8} \multirow{-2}{*}{4} &
      \multicolumn{1}{c|}{$\enernm{u-u_h}$} &
      \multicolumn{1}{c|}{7.55e-01} & 7.29e-02 & 3.12e-03 & 2.06e-04 &
      1.13e-05 & 4.05 \\ \hline

      & $\nm{u-u_h}{L^2}$ & 1.27e-02 & 3.31e-04 & 2.95e-06 & 3.84e-08
      & 5.62e-10 & 6.19 \\ \cline{2-8} \multirow{-2}{*}{5} &
      \multicolumn{1}{c|}{$\enernm{u-u_h}$} &
      \multicolumn{1}{c|}{5.73e-01} & 3.54e-02 & 7.10e-04 & 1.88e-05 &
      5.23e-07 & 5.10 \\ \hline

      & $\nm{u-u_h}{L^2}$ & 1.76e-02 & 2.09e-04 & 1.21e-06 & 8.66e-09
      & 2.99e-11 & 7.28 \\ \cline{2-8} \multirow{-2}{*}{6} &
      \multicolumn{1}{c|}{$\enernm{u-u_h}$} &
      \multicolumn{1}{c|}{1.10e+00} & 2.07e-02 & 3.36e-04 & 4.74e-06 &
      2.97e-08 & 6.24 \\ \hline
  \end{tabular}}
  \caption{Errors and convergence rates for Example 3.}
  \label{mixed_mesh}
  \end{center}
\end{table}
\section{Conclusions}~\label{sec:conclusion}
Using a least-squares patch reconstruction, we construct a new
discontinuous finite element space on polygonal mesh, which together
with the variational formulation of DG method
gives a new approximation method for partial differential equations. A
novelty of this method is that arbitrary-order accuracy has been achieved with
only one degree of freedom on each element, while the shape of the
element may be arbitrary. Optimal error estimates have been proved and
a variety of numerical examples demonstrate the superior performance
of the method. It would be interesting to consider the $h-m$ version of the proposed method and the corresponding adaptive
refinement strategy, and to consider the choice of the interior
penalty parameter that allow for edge/face degeneration as in~\cite{cangiani2014hp}, which is key for implementation the method on more general polytopic mesh. Moreover, the assumption on the element patch may be
further weakened, which may render more flexibility for the method. We shall leave all these issues to further exploration.
\section*{Acknowledgment}
The authors would like to thank Dr. Fengyang Tang for his help in the
earlier stage of the present work.
%

\end{document}